\documentclass[a4paper,twoside,11pt]{amsart}

\usepackage[a4paper]{geometry}
\usepackage[latin1]{inputenc}
\usepackage[T1]{fontenc}

\usepackage{lipsum} 
\usepackage{comment}
\usepackage{color}

\usepackage[english,francais]{babel}

\usepackage{amssymb,amsmath,amsthm,amscd}
\usepackage{mathrsfs}
\usepackage{stmaryrd}

\usepackage{subfig}

\usepackage{tabularx} 
\usepackage{calc} 
\usepackage[pdftex]{graphicx} 
\usepackage[all]{xy}
\xyoption{v2}
\xyoption{2cell}
\UseAllTwocells

\usepackage[pdfauthor={Christophe Prange},pdftitle={First-order expansion for the Dirichlet eigenvalues of an elliptic system with oscillating coefficients},pdftex]{hyperref}

\numberwithin{equation}{section}

\DeclareMathOperator{\reel}{Re}
\DeclareMathOperator{\spectre}{Sp}

\let\oldmarginpar\marginpar
\renewcommand\marginpar[1]{\oldmarginpar{\color{red}\raggedleft\tiny #1}}

\title{First-order expansion for the Dirichlet eigenvalues of an elliptic system with oscillating coefficients}
\author{Christophe Prange$^*$}
\thanks{$^*$ Institut Math\'ematique de Jussieu, $175$ rue du Chevaleret, $75013$ Paris, France\\
\emph{E-mail address:} \texttt{prange@math.jussieu.fr}}

\begin{document}

\selectlanguage{english}

\newtheorem{theo}{Theorem}[section]
\newtheorem{prop}[theo]{Proposition}
\newtheorem{lem}[theo]{Lemma}
\newtheorem{cor}[theo]{Corollary}
\newtheorem*{theo*}{Theorem}

\theoremstyle{definition}
\newtheorem{defi}[theo]{Definition}

\theoremstyle{remark}
\newtheorem{rem}[theo]{Remark}

\title{First-order expansion for the Dirichlet eigenvalues of an elliptic system with oscillating coefficients}

\begin{abstract}
This paper is concerned with the homogenization of the Dirichlet eigenvalue problem, posed in a bounded domain $\Omega\subset\mathbb R^2$, for a vectorial elliptic operator $-\nabla\cdot A^\varepsilon(\cdot)\nabla$ with $\varepsilon$-periodic coefficients. We analyse the asymptotics of the eigenvalues $\lambda^{\varepsilon,k}$ when $\varepsilon\rightarrow 0$, the mode $k$ being fixed. A first-order asymptotic expansion is proven for $\lambda^{\varepsilon,k}$ in the case when $\Omega$ is either a smooth uniformly convex domain, or a convex polygonal domain with sides of slopes satisfying a small divisors assumption. Our results extend those of Moskow and Vogelius in \cite{moscovog} restricted to scalar operators and convex polygonal domains with sides of rational slopes. We take advantage of the recent progress due to G\'erard-Varet and Masmoudi \cite{dgvnm,dgvnm2} in the homogenization of boundary layer type systems.
\end{abstract}

\maketitle




\selectlanguage{english}

\pagestyle{plain}

\section{Introduction}
This paper is devoted to the homogenization of the Dirichlet eigenvalue problem
\begin{equation}\label{syseig}
\left\{
\begin{array}{rll}
-\nabla \cdot A\bigl(\frac{x}{\varepsilon}\bigr)\nabla v^\varepsilon=&\lambda^\varepsilon v^\varepsilon,& x\in \Omega\\
v^\varepsilon=&0,& x\in \partial \Omega
\end{array}
\right. 
\end{equation}
posed in a planar domain $\Omega$ with periodic microstructure.
Some reasons for the study of the asymptotical behaviour of the eigenvalues when the period $\varepsilon\rightarrow 0$  are expounded in \cite{santovog}. Among physical motivations is the analysis of low frequency vibrations in periodic composite media. Significant progress in the direction of a better understanding of the asymptotics of $\lambda^\varepsilon$ when $\varepsilon\rightarrow 0$ has been achieved first by Santosa and Vogelius in \cite{santovog} then by Moskow and Vogelius in \cite{moscovog} under weaker assumptions. Our work extends the results of \cite{moscovog} to the case of elliptic systems and more general domains $\Omega$. Moreover, error estimates have been improved.

Before entering into more details, let us state our mathematical framework. Let $N\in\mathbb N$, $N\geq 1$. Throughout this paper, $\Omega$ stands for a bounded open subset of $\mathbb R^2$, $v^\varepsilon=v^\varepsilon(x)\in\mathbb R^N$ and $A=A^{\alpha\beta}(y)\in M_N(\mathbb R)$ is a family of periodic functions of $y\in\mathbb T^2$ indexed by $1\leq\alpha,\beta\leq 2$. Therefore, taking advantage of Einstein's convention for summation:
\begin{equation*}
\biggl(\nabla\cdot A\Bigl(\frac{x}{\varepsilon}\Bigr)\nabla v^\varepsilon\biggl)_i=\partial_{x_\alpha}\biggl(A^{\alpha\beta}_{ij}\Bigl(\frac{x}{\varepsilon}\Bigr)\partial_{x_\beta}v^\varepsilon_j\biggr).
\end{equation*}
All along these lines, $C>0$ denotes an arbitrary constant independant of $\varepsilon$. The main assumptions on $A$ are:
\begin{description}
\item[(A1) ellipticity] there exists $\lambda>0$ such that for all $\xi=\left(\xi^1,\xi^2\right)\in\mathbb R^N\times\mathbb R^N$, for all $y\in\mathbb R^2$,
\begin{equation*}
\lambda\xi^\alpha\cdot\xi^\alpha\leq A^{\alpha\beta}(y)\xi^\alpha\cdot\xi^\beta\leq \lambda^{-1}\xi^\alpha\cdot\xi^\alpha;
\end{equation*}
\item[(A2) periodicity] for all $y\in\mathbb R^2$, for all $h\in \mathbb Z^2$,
\begin{equation*}
A(y+h)=A(y);
\end{equation*}
\item[(A3) regularity] $A$ is supposed to belong to $C^\infty(\overline{\Omega})$;
\item[(A4) symmetry] for all $1\leq\alpha,\beta\leq 2$, for all $1\leq i,j\leq N$, $A^{\alpha\beta}_{ij}=A^{\beta\alpha}_{ji}$.
\end{description}
Unless otherwise specified, we always assume {\bf (A1)}, \dots, {\bf (A4)}.
In a very classical fashion, boundedness of $\Omega$ and ellipticity of $A$ imply, through Poincar\'e inequality and Lax Milgram lemma, that the linear mapping
\begin{equation*}
T^\varepsilon:\quad f\in L^2(\Omega)\longmapsto u^\varepsilon\in H^1_0(\Omega),
\end{equation*}
where $u^\varepsilon$ is the unique weak solution of \eqref{syseig} with r.h.s. equal to $f$, is well defined, continuous and injective. If one composes $T^\varepsilon$ with the compact injection of $H^1_0(\Omega)$ in $L^2(\Omega)$, one gets a compact operator, again denoted by $T^\varepsilon$, from $L^2(\Omega)$ in itself. Assumption {\bf (A4)} tells that $T^\varepsilon$ is self-adjoint. 

From the previous considerations, we know that our eigenvalue problem \eqref{syseig} is well posed. There exists a sequence of eigenvalues $0<\lambda^{\varepsilon,0}\leq \lambda^{\varepsilon,1}\leq \ldots\lambda^{\varepsilon,k}\xrightarrow{k\rightarrow\infty}\infty$ 
and a hilbertian basis $(v^{\varepsilon,k})$ of $L^2(\Omega)$ of corresponding eigenvectors. To tackle the issue of the asymptotical behaviour of $\left(\lambda^\varepsilon,v^\varepsilon\right)$ we deeply use the periodic structure of the problem at microscale contained in {\bf (A2)}. We expand, at least formally, $v^\varepsilon$ and $\lambda^\varepsilon$ in powers of $\varepsilon$
\begin{align}
v^\varepsilon(x)&\approx v^0\Bigl(x,\frac{x}{\varepsilon}\Bigr)+\varepsilon v^1\Bigl(x,\frac{x}{\varepsilon}\Bigr)+\varepsilon^2v^2\Bigl(x,\frac{x}{\varepsilon}\Bigr)+\ldots\label{dvptveps}\\
\lambda^\varepsilon&\approx\lambda^0+\varepsilon\lambda^1+\varepsilon^2\lambda^2+\ldots\label{dvptleps}
\end{align}
where for all $i\in\mathbb N$, $v^i=v^i(x,y)$ is periodic in the $y\in\mathbb T^2$ variable. Plugging \eqref{dvptveps} and \eqref{dvptleps} in \eqref{syseig} and identifying the powers of $\varepsilon$ yields that $v^0$ does not depend on $y$ and that $(\lambda^0,v^0)$ solves the homogenized eigenvalue problem
\begin{equation}\label{syseighom}
\left\{
\begin{array}{rll}
-\nabla \cdot A^0\nabla v^0=&\lambda^0v^0,& x\in \Omega\\
v^0=&0,& x\in \partial \Omega
\end{array}
\right. .
\end{equation}
As usual, the constant homogenized tensor $A^0=A^{0,\alpha\beta}\in M_N(\mathbb R)$ in \eqref{syseighom} is given by
\begin{equation*}
A^{0,\alpha\beta}:=\int_{\mathbb T^2}A^{\alpha\beta}(y)dy+\int_{\mathbb T^2}A^{\alpha\gamma}(y)\partial_{y_\gamma}\chi^\beta(y)dy,
\end{equation*}
where the family $\chi=\chi^\gamma(y)\in M_N(\mathbb R)$, $y\in\mathbb T^2$, solves the cell problem
\begin{equation}\label{eqdefchi}
-\nabla_y \cdot A(y)\nabla_y \chi^\gamma=\partial_{y_\alpha}A^{\alpha\gamma},\ y\in \mathbb T^2\qquad \mbox{and}\qquad \int_{\mathbb T^2}\chi^\gamma(y)dy=0.
\end{equation}
Note that $A^0$ fulfils assumptions {\bf (A1)} and {\bf (A4)}, so there exists a sequence of eigenvalues $0<\lambda^{0,0}\leq \lambda^{0,1}\leq \ldots\lambda^{0,k}\xrightarrow{k\rightarrow\infty}\infty$ and a hilbertian basis $(v^{0,k})$ of $L^2(\Omega)$ of corresponding eigenvectors. Let $T^0$ denote the operator similar to $T^\varepsilon$ with $A^0$ in place of $A^\varepsilon:=A\left(\frac{\cdot}{\varepsilon}\right)$.

\subsection{What is at stake?}
Let us now focus on the convergence properties of the eigenvalues $\lambda^{\varepsilon,k}$ when $\varepsilon\rightarrow\infty$ and the mode $k$ is fixed. The first thing we know from \cite{allconca}, among other papers, is that for all $k\in\mathbb N$,
\begin{equation}\label{ineqvpnormeT}
\left|
\frac{1}{\lambda^{\varepsilon,k}}-\frac{1}{\lambda^{0,k}}
\right|\leq
\begin{Vmatrix}
T^\varepsilon-T^0
\end{Vmatrix}_{\mathcal L (L^2(\Omega))}
\end{equation}
and that
\begin{equation*}
T^\varepsilon\xrightarrow{\varepsilon\rightarrow 0}T^0
\end{equation*}
in $\mathcal L\bigl(L^2(\Omega)\bigr)$ norm. Therefore, for all $k\in\mathbb N$,
\begin{equation*}
\lambda^{\varepsilon,k}\xrightarrow{\varepsilon\rightarrow 0}\lambda^{0,k}
\end{equation*}
no matter wether $\lambda^{0,k}$ is simple or not. When $N=1$, i.e. the system \eqref{syseig} is a scalar equation, on condition that one has enough regularity on $v^{0,k}$ ($v^{0,k}\in H^2(\Omega)$ is sufficient) so that an estimate like
\begin{equation}\label{estuepsu0epsL2}
\begin{Vmatrix}
v^{\varepsilon,k}(x)-v^{0,k}(x)
\end{Vmatrix}_{L^2(\Omega)}\leq C\varepsilon
\begin{Vmatrix}
v^{0,k}
\end{Vmatrix}_{H^{2}(\Omega)}
\end{equation}
holds, one has in addition the error estimate
\begin{equation}\label{estvpOeps}
\begin{vmatrix}
\lambda^{\varepsilon,k}-\lambda^{0,k}
\end{vmatrix}\leq
C_k\varepsilon.
\end{equation}

Estimate \eqref{estvpOeps} is the starting point of the work of Moskow and Vogelius. Indeed, it leads to natural questions, such that:
\begin{enumerate}
\item What are the limit points of 
\begin{equation}\label{limpointsquot}
\frac{\lambda^{\varepsilon,k}-\lambda^{0,k}}{\varepsilon}
\mbox{?}
\end{equation}
\item Is there possibly one unique limit point?
\item What is the next term in the asymptotic expansion?
\end{enumerate}

When it does not lead to any confusion, 
we shall now omit the exponent $k$:
\begin{align*}
&\lambda^\varepsilon:=\lambda^{\varepsilon,k}\quad (\mbox{resp. } \lambda^0:=\lambda^{0,k})\\
&v^\varepsilon:=v^{\varepsilon,k}\quad (\mbox{resp. } v^0:=v^{0,k}). 
\end{align*}

In \cite{moscovog}, Moskow and Vogelius get an asymptotic formula for the eigenvalue $\lambda^{\varepsilon}$ of the scalar equation ($N=1$), valid up to the order $1$ in $\varepsilon$, provided that $v^{0}\in H^2(\Omega)$, which is actually true for sufficiently smooth domains $\Omega$ (convex or $C^2$ domains). They fully describe the first-order corrections, i.e. the limit points of \eqref{limpointsquot}, in the case when $\Omega$ is a convex polygonal domain with sides of rational slopes. In this case there is a continuum of accumulation points for \eqref{limpointsquot}.
\begin{theo}[Moskow and Vogelius in \cite{moscovog}]\label{theomoscovog}
Assume that $\Omega$ is a convex polygonal domain with sides of rational slopes and that $N=1$. Assume furthermore that $\lambda^0$ is a simple eigenvalue.\\
Then, there exists $\vartheta^*_{bl}\in L^2(\Omega)$ solving an explicit homogenized elliptic boundary value problem (see \eqref{sysoscbordhom}), and a sequence $(\varepsilon_n)$ tending to $0$ such that
\begin{equation}\label{firstorderMV}
\lambda^{\varepsilon_n}=\lambda^0+\varepsilon_n\lambda^0\int_{\Omega}\vartheta^*_{bl}(x)v^0(x)dx+o(\varepsilon_n).
\end{equation}
\end{theo}

\subsection{Difficulties and strategy}

One faces essentially two kind of difficulties in proving a first-order asymptotic expansion for the eigenvalues like \eqref{firstorderMV}: the first is linked with the homogenization of boundary layer type systems, the second has to do with the regularity of solutions to elliptic systems in nonsmooth domains like polygons. We sketch how these difficulties are addressed by Moskow and Vogelius and how we extend their results to the case of elliptic systems and more general polygonal or smooth domains $\Omega$. 

\subsubsection{Homogenization of boundary layer systems}
While $v^0$ solves \eqref{syseighom} with a homogeneous Dirichlet boundary condition on $\partial\Omega$, $v^1\left(\cdot,\frac{\cdot}{\varepsilon}\right)$ does not cancel in general on $\partial\Omega$. For this reason, the formal asymptotic expansion \eqref{dvptveps} is inadequate to establish a first-order expansion for the eigenvalues. Boundary layers need to be taken into account. Considering
\begin{equation*}\label{devasybl}
v^\varepsilon(x)\approx v^0(x)+\varepsilon\left[v^1\Bigl(x,\frac{x}{\varepsilon}\Bigr)+v_{bl}^{1,\varepsilon}(x)\right]+\varepsilon^2\left[v^2\Bigl(x,\frac{x}{\varepsilon}\Bigr)+v_{bl}^{2,\varepsilon}(x)\right]+\ldots
\end{equation*}
where $v_{bl}^{i,\varepsilon}$ solves 
\begin{equation}\label{sysoscbordv^i_bl}
\left\{
\begin{array}{rll}
-\nabla \cdot A\bigl(\frac{x}{\varepsilon}\bigr)\nabla v_{bl}^{i,\varepsilon}=&0,& x\in \Omega\\
v_{bl}^{i,\varepsilon}=&-v^i\bigl(x,\frac{x}{\varepsilon}\bigr),& x\in \partial \Omega
\end{array}
\right.
\end{equation}
proves to be more relevant than \eqref{dvptveps}. Moreover, $\vartheta^*_{bl}$, appearing in \eqref{firstorderMV}, comes from the homogenization of an elliptic boundary layer system like \eqref{sysoscbordv^i_bl}.

The heart of the proof of theorem \ref{theomoscovog} is subsequently the homogenization of boundary layer type systems
\begin{equation}\label{sysoscbord}
\left\{
\begin{array}{rll}
-\nabla \cdot A\bigl(\frac{x}{\varepsilon}\bigr)\nabla u_{bl}^\varepsilon=&0,& x\in \Omega\\
u_{bl}^\varepsilon=&\varphi\bigl(x,\frac{x}{\varepsilon}\bigr),& x\in \partial \Omega
\end{array}
\right.
\end{equation}
with $\varphi=\varphi(x,y):=\Phi(y)\varphi_0(x)$, $\Phi$ being, unless stated otherwise, a smooth function on $\mathbb T^2$ and $\varphi_0\in H^{\frac{1}{2}}(\partial\Omega)$.

Compared to $u^\varepsilon$ solution of
\begin{equation}\label{sysoscbis}
\left\{
\begin{array}{rll}
-\nabla \cdot A\bigl(\frac{x}{\varepsilon}\bigr)\nabla u^\varepsilon=&f,& x\in \Omega\\
u^\varepsilon=&\varphi_0,& x\in \partial \Omega
\end{array}
\right.
\end{equation}
whose homogenization is now a classical topic, the analysis of the asymptotics of $u_{bl}^\varepsilon$ when $\varepsilon\rightarrow 0$ is complicated by the oscillating boundary data in \eqref{sysoscbord} for at least two reasons:
\begin{enumerate}
\item 
We lack uniform a priori estimates for $u_{bl}^\varepsilon$ in $H^1(\Omega)$ norm. This is due to the fact that
\begin{equation*}
\begin{Vmatrix}
\varphi\bigl(x,\frac{x}{\varepsilon}\bigr)
\end{Vmatrix}_{H^{\frac{1}{2}}(\partial\Omega)}=O\bigl(\varepsilon^{-\frac{1}{2}}\bigr).
\end{equation*}
\item 
The behaviour of the boundary layer along the boundary $\partial\Omega$ deeply depends on the interaction between the periodic lattice and the boundary. Thus one can not expect in general periodicity of the boundary layer along the boundary.
\end{enumerate}

The proofs of Moskow and Vogelius intensively rely on a result due to Avellaneda and Lin, which addresses a priori estimates for elliptic equations in domains $\Omega$ with quite low regularity.
\begin{theo}[Avellaneda and Lin in \cite{alinscal} theorem $3$]\label{theoavlinscal}
Assume that $\Omega$ is Lipschitz and satisfies a uniform exterior sphere condition. Assume furthermore that $N=1$.\\
Then, for all $1<p<\infty$, there exists $C>0$ such that for all boundary data function $\varphi\bigl(\cdot,\frac{\cdot}{\varepsilon}\bigr)\in L^p(\partial\Omega)$, there is a solution $u_{bl}^\varepsilon\in L^p(\Omega)$ of \eqref{sysoscbord} satisfying
\begin{equation}\label{estavlinN=1}
\begin{Vmatrix}
u_{bl}^\varepsilon
\end{Vmatrix}_{L^p(\Omega)}\leq
C\begin{Vmatrix}
\varphi\bigl(\cdot,\frac{\cdot}{\varepsilon}\bigr)
\end{Vmatrix}_{L^p(\partial\Omega)}.
\end{equation}
\end{theo}
Its usefulness for our boundary layer system \eqref{sysoscbord} comes from the following simple remark: $\varphi\bigl(\cdot,\frac{\cdot}{\varepsilon}\bigr)$ is bounded in $L^2(\partial\Omega)$ norm, but not in $H^{\frac{1}{2}}(\partial\Omega)$ norm. Consequently, 
$\begin{Vmatrix}
u_{bl}^\varepsilon
\end{Vmatrix}_{L^2(\Omega)}=O(1)$. An estimate similar to \eqref{estavlinN=1} holds also for elliptic systems, yet under stronger regularity assumptions on $\Omega$.
\begin{theo}[Avellaneda and Lin in \cite{alin} theorem $3$] 
\label{theoavlinsys}
Let $N$ be any positive integer. Assume that $\Omega$ is $C^{1,\alpha}$ with $0<\alpha\leq 1$.\\
Then, the conclusion of theorem \ref{theoavlinscal} remains true.
\end{theo}
This theorem can be applied when $\Omega$ is smooth, but due to its strong regularity assumption on $\Omega$, it is useless in the case when $\Omega$ is a polygonal domain. A precise analysis of the boundary layer system is needed in this case. Beyond the results of Avellaneda and Lin, 
the analysis of \eqref{sysoscbord} has been carried out in the context of:
\begin{enumerate}
\item convex polygonal domains $\Omega$, first with edges of rational slopes by Moskow and Vogelius in \cite{moscovog}, Allaire and Amar in \cite{allam} (scalar case), then with edges of slopes satisfying a generic small divisors assumption by G\'erard-Varet and Masmoudi in \cite{dgvnm};
\item smooth domains with uniformly convex boundary by G\'erard-Varet and Masmoudi in the recent paper \cite{dgvnm2}.
\end{enumerate}
This recent progress in the homogenization of \eqref{sysoscbord}, due to G\'erard-Varet and Masmoudi, opens the way to our generalizations.

\subsubsection{Regularity}Besides the issue of the homogenization of boundary layer systems comes the problem of regularity. Regularity is required in order to carry out the energy estimates of the paper. Of course this is only a problem when $\Omega$ is a polygonal domain; if $\Omega$ is smooth, all functions we deal with belong to $C^\infty(\overline{\Omega})$. 

Assume now that $\Omega$ is a convex polygon. In the scalar case the results of Grisvard in \cite{gri} (theorem $3.2.1.2$) and \cite{gris2} (section $2.7$), recalled in \cite{moscovog}, yield that $v^0\in H^2(\Omega)$, because of convexity. We even know better. Indeed $v^0$ solves \eqref{syseighom} with r.h.s. $\lambda^0v^0\in H^1(\Omega)$. Therefore, $v^0\in H^{2+\omega}(\Omega)$, with $0<\omega$. 

This $H^{2+\omega}(\Omega)$ regularity on $v^0$ appears to be the minimal regularity one has to assume in order to get a first-order expansion like \eqref{firstorderMV}. It is a corollary of the work of Dauge \cite{dauge88} on the one hand, and Kozlov, Maz'ya and Rossmann \cite{KMR01} on the other hand, that the results of Grisvard extend to systems with constant coefficients. More precisely:
\begin{theo}\label{theoregdaugekmr}
Let $N\geq 1$ and $u^0\in H^1_0(\Omega)$ be the unique solution of \eqref{syseighom} with r.h.s. equal to $f\in H^{-1}(\Omega)$. Assume that $\Omega$ is a convex polygonal domain.
\begin{enumerate}
\item If $f\in H^{-1+\omega}(\Omega)$ with $0<\omega< 1$ and $\omega\neq\frac{1}{2}$, then $u^0\in H^{1+\omega}(\Omega)$.
\item If $f\in L^2(\Omega)$, then $u^0\in H^2(\Omega)$.
\item If $f\in H^1(\Omega)$, then there exists $0<\omega\leq 1$ such that $u^0\in H^{2+\omega}(\Omega)$.
\end{enumerate}
\end{theo}

Let us give a sketch of how to deduce such regularity statements from \cite{dauge88} and \cite{KMR01} (see these references for more details). We know from \cite{dauge88} (see lemma $5.7$) that for all $s>0$, for all vertex $x\in\partial\Omega$
\begin{equation*}
\{0<\reel(\lambda)<s\}\cap\spectre\mathcal L_x\; \mbox{is finite,}
\end{equation*}
where $\spectre\mathcal L_x$ denotes the spectrum of a pencil associated to our problem at vertex $x$. Besides, Kozlov, Maz'ya and Rossmann prove in \cite{KMR01}, theorem $8.6.2$, that for strongly elliptic systems, with constant coefficients, satisfying the symmetry assumption {\bf (A4)}, posed in the convex polygonal domain $\Omega$, 
\begin{equation*}
\left\{0\leq\reel(\lambda)\leq 1\right\}\cap\spectre\mathcal L_x=\varnothing
\end{equation*}
for all vertex $x$. This fact collapses if $\Omega$ has at least one angle $\geq\pi$. Yet $\Omega$ being polygonal and convex, it follows from \cite{dauge88}, in particular paragraph $7.16$, corollary $5.16$ and theorem $5.5$, that the operator
\begin{equation*}
L^{(s)}:\quad u\in H^{s+1}(\Omega)\cap H^1_0(\Omega)\longmapsto -\nabla\cdot A\Bigl(\frac{x}{\varepsilon}\Bigr)\nabla u\in H^{s-1}(\Omega)
\end{equation*}
is a Fredholm operator for all $0\leq s\neq\frac{1}{2}$ satisfying $\{\reel(\lambda)=s\}\cap\spectre\mathcal L_x=\varnothing$  for all vertex $x\in\partial\Omega$. At this point, one deduces that $L^{(s)}$ is a Fredholm operator for all $0\leq s\leq 1$, $s\neq\frac{1}{2}$. Moreover, there exists $0<\omega\leq 1$ such that $L^{(1+\omega)}$ is a Fredholm operator.

Let $0\leq s$ such that $L^{(s)}$ is a Fredholm operator and let $f\in H^{s-1}(\Omega)$. Two situations are possible. If there is a vertex $x\in\partial\Omega$ and $\lambda\in\{0<\reel(\lambda)<s\}\cap\spectre\mathcal L_x$, then theorem $5.11$ in \cite{dauge88} yields the existence of $u^0_{reg}\in H^{1+s}(\Omega)$, the regular part, and $u^0_{sing}\in H^{1+\gamma}(\Omega)$, the singular part, with $0<\gamma<\min_{\lambda\in\{0<\reel(\lambda)<s\}\cap\bigcup_{x}\spectre\mathcal L_x}\reel(\lambda)$, such that
\begin{equation*}
u^0=u^0_{sing}+u^0_{reg}\in H^{1+\gamma}(\Omega).
\end{equation*}
On the contrary, if for all vertex $x$, $\{0<\reel(\lambda)<s\}\cap\spectre\mathcal L_x=\varnothing$, then $u^0=u^0_{reg}$ is in $H^{1+s}(\Omega)$. The two first points of theorem \ref{theoregdaugekmr} as well as the third now easily follow from the preceding results.


Each point of theorem \ref{theoregdaugekmr} plays a role in our reasoning. One can alternatively invoke weaker regularity results such as:
\begin{theo}[Agranovich in \cite{agra07} theorem $1$]\label{agra07th}
Assume that $\Omega$ is a Lipschitz domain. Let $u^\varepsilon\in H^1_0(\Omega)$ be the unique variational solution of \eqref{syseig} with r.h.s. equal to $f\in H^{-1}(\Omega)$. Assume furthermore that $f\in H^{-1+\omega}(\Omega)$ with $0\leq \omega<\frac{1}{2}$.\\
Then, $u^\varepsilon\in H^{1+\omega}(\Omega)$.
\end{theo}

\subsection{Outline of our results}

This article answers relevant questions asked by Moskow and Vogelius. Quoting \cite{moscovog} (section $5$):
\begin{quote}
It should be extremely interesting to derive a similar representation formula for polygons with sides of irrational slopes or for smooth domains. In particular, it would be interesting to see if this leads to a single first-order correction for any eigenvalue.  
\end{quote}

We manage to free ourselves from the assumption $N=1$. 
Our main results then sum up in the upcoming theorems. We treat separately two different classes of domains $\Omega$: on the one hand very smooth domains, on the other hand convex polygonal domains. \emph{All the definitions we use are made rigourous later in the paper.}

Assume that $\lambda^0$ is an eigenvalue of order $m$. Let $\lambda^0=\lambda^{0,k}=\lambda^{0,k+1}=\ldots=\lambda^{0,k+m-1}$ be the eigenvalues repeated with multiplicity. We call $E_{\lambda^0}$ the finite-dimensional eigenspace associated to the eigenvalue $\lambda^0$. Note that the eigenvectors $v^{0,k},\ldots,v^{0,k+m-1}$ form an orthogonal basis of $E_{\lambda^0}$.

Our first theorem is concerned with smooth domains $\Omega$.
\begin{theo}\label{theoasylisse}
Assume that $\Omega$ is a smooth $C^\infty$ bounded domain with uniformly convex boundary.\\
Then, for every $0\leq j\leq m-1$, there exists a unique $\vartheta_{j,bl}^*\in L^2(\Omega)$ such that for all $0\leq\gamma<\frac{1}{11}$,
\begin{equation}\label{devptasylisse}
\left[\frac{1}{m}\sum_{j=0}^{m-1}\frac{1}{\lambda^{\varepsilon,k+j}}\right]^{-1}=\lambda^0
+\varepsilon\frac{\lambda^0}{m}\sum_{j=0}^{m-1}\int_{\Omega}\vartheta_{j,bl}^*(x)\cdot v^{0,k+j}(x)dx+O\bigl(\varepsilon^{1+\gamma}\bigr).
\end{equation}
\end{theo}

The next theorem faces the same problem for convex polygonal domains $\Omega$. 
\begin{theo}\label{theoasypol}
Assume that $\Omega$ is a convex polygonal domain with sides of slopes satisfying a generic small divisors assumption; see section \ref{convexpol}.
\begin{enumerate}
\item Then for every $0\leq j\leq m-1$, there exists a unique $\vartheta_{j,bl}^*\in L^2(\Omega)$ and $0<\gamma$ such that 
\begin{equation}\label{devptasypol}
\left[\frac{1}{m}\sum_{j=0}^{m-1}\frac{1}{\lambda^{\varepsilon,k+j}}\right]^{-1}=\lambda^0
+\varepsilon\frac{\lambda^0}{m}\sum_{j=0}^{m-1}\int_{\Omega}\vartheta_{j,bl}^*(x)\cdot v^{0,k+j}(x)dx+O\bigl(\varepsilon^{1+\gamma}\bigr).
\end{equation}
\item If $E_{\lambda^0}\subset H^{3}(\Omega)\cap C^2(\overline{\Omega})$, then for every $0\leq j\leq m-1$, there exists a unique $\vartheta_{j,bl}^*\in L^2(\Omega)$ such that
\begin{equation}\label{devptasypolreg}
\left[\frac{1}{m}\sum_{j=0}^{m-1}\frac{1}{\lambda^{\varepsilon,k+j}}\right]^{-1}=\lambda^0
+\varepsilon\frac{\lambda^0}{m}\sum_{j=0}^{m-1}\int_{\Omega}\vartheta_{j,bl}^*(x)\cdot v^{0,k+j}(x)dx+O\bigl(\varepsilon^{\frac{3}{2}}\bigr).
\end{equation}
\end{enumerate}
\end{theo}

We stress that theorem \ref{theoasypol} has two parts. The first point is a general result: due to the assumptions on $A$ (in particular {\bf (A1)} and {\bf (A4)}) and on $\Omega$ (polygonal and convex), the $H^{2+\omega}(\Omega)$ regularity, with $0<\omega$, needed on the eigenvectors for the proof, happens to be automatically fulfilled. The second part of the theorem states an optimal result in view of our proof, in terms of convergence rate, but needs to assume more regularity on the eigenfunctions.

There is a analog of theorem \ref{theoasypol} in the case of a convex polygon with sides of rational slopes, which improves theorem \ref{theomoscovog}. Estimate \eqref{devptasypol} (resp. \eqref{devptasypolreg}) still holds however up to the extraction of a subsequence $(\varepsilon_n)$. Throughout the paper, we indicate how to adapt the proofs to this case. 

When $\lambda^0$ is simple, \eqref{devptasylisse}, \eqref{devptasypol} and \eqref{devptasypolreg} yield the first-order expansion
\begin{equation*}
\lambda^{\varepsilon}=\lambda^0+\varepsilon\lambda^0\int_{\Omega}\vartheta_{bl}^*(x)v^0(x)dx+o\bigl(\varepsilon^{1+\gamma}\bigr).
\end{equation*}
valid for appropriate exponents $\gamma$.

A consequence of these two theorems \ref{theoasylisse} and \ref{theoasypol} is that the first-order correction to the eigenvalue $\lambda^0$ is identified and unique. Furthermore, it appears in the course of the proof of theorem \ref{theoasypol} that $\vartheta_{j,bl}^*$ is a solution of an homogenized elliptic boundary value problem (see \eqref{sysoscbordhomdiv}), whose data can be made explicit. It thus opens the door to numerical computations.

\subsection{Organization of the paper}
In section \ref{secsomerrorest}, we prove some corrector results for $u^\varepsilon$ solution of \eqref{sysoscbis}. Such estimates do exist in the litterature, but we focus on minimal regularity assumptions. In particular, we extend the bounds of \cite{moscovog} to elliptic systems (non necessarily symmetric) and get new ones, which are useful in the rest of the paper. Section \ref{sechomblsys} is devoted to the homogenization of boundary layer type systems. We analyse the convergence in $L^2(\Omega)$ of $\vartheta^\varepsilon_{v,bl}$ solution of \eqref{sysoscbord} with $\varphi(x,y):=-\chi^\alpha(y)\partial_{x_\alpha}v^0(x)$, in the two different settings: $\Omega$ is a smooth domain with uniformly convex boundary or a convex polygonal domain with the additional diophantine condition on the slopes. This work is the central step in the proof of theorems \ref{theoasylisse} and \ref{theoasypol}. 
The final step is done in section \ref{secasy}, where a first-order correction formula for $\lambda^\varepsilon$, in terms of the limit of $\vartheta^\varepsilon_{v,bl}$ when $\varepsilon\rightarrow 0$, is obtained.

\selectlanguage{english}


\section{Some error estimates}
\label{secsomerrorest}
Let $f\in H^{-1}(\Omega)$ and $\varphi_0\in H^\frac{1}{2}(\partial\Omega)$. The solution $u^\varepsilon$ of 
\begin{equation}\label{sysosc}
\left\{
\begin{array}{rll}
-\nabla \cdot A\bigl(\frac{x}{\varepsilon}\bigr)\nabla u^\varepsilon=&f,& x\in \Omega\\
u^\varepsilon=&\varphi_0,& x\in \partial \Omega
\end{array}
\right.
\end{equation}
exists, is unique in $H^1(\Omega)$ and converges strongly in $L^2(\Omega)$ towards $u^0\in H^1(\Omega)$ solving the elliptic system
\begin{equation}\label{sysu0}
\left\{
\begin{array}{rll}
-\nabla \cdot A^0\nabla u^0=&f,& x\in \Omega\\
u^0=&\varphi_0,& x\in \partial \Omega
\end{array}
\right. .
\end{equation}
We focus here on estimates in norm showing how fast this convergence takes place. We do not need assumption {\bf (A4)}, i.e. the symmetry of $A$.

\subsection{Multiscale expansions}\label{secmultscal}Before coming to the estimates, let us recall some basic facts about multiscale expansions. In the same fashion as $v^\varepsilon$ (see \eqref{dvptveps}), we expand $u^\varepsilon$:
\begin{equation}\label{dvptueps}
u^\varepsilon(x)\approx u^0\Bigl(x,\frac{x}{\varepsilon}\Bigr)+\varepsilon u^1\Bigl(x,\frac{x}{\varepsilon}\Bigr)+\varepsilon^2u^2\Bigl(x,\frac{x}{\varepsilon}\Bigr)+\ldots
\end{equation}
Plugging \eqref{dvptueps} in \eqref{sysosc} and identifying the powers of $\varepsilon$ yields, at least formally, 
\begin{enumerate}
\item that $u^0$ solves the homogenized system 
\begin{equation*}
\left\{
\begin{array}{rll}
-\nabla \cdot A^0\nabla u^0=&f,& x\in \Omega\\
u^0=&\varphi_0,& x\in \partial \Omega
\end{array}
\right. ,
\end{equation*}
\item that $u^1=u^1(x,y):=\chi^\alpha(y)\partial_{x_\alpha}u^0(x)+\bar{u}^1(x)$ where $\chi^\alpha$ is the function defined in \eqref{eqdefchi},
\item and that $u^2=u^2(x,y):=\Gamma^{\alpha\beta}(y)\partial_{x_\alpha}\partial_{x_\beta}u^0(x)+\chi^\alpha(y)\partial_{x_\alpha}\bar{u}^1(x)+\bar{u}^2(x)$ where $\Gamma^{\alpha\beta}$ solves
\begin{equation*}
-\nabla_y \cdot A(y)\nabla_y \Gamma^{\alpha\beta}=B^{\alpha\beta}-\int_{\mathbb T^2}B^{\alpha\beta}(y)dy,\ y\in \mathbb T^2\qquad \mbox{and}\qquad \int_{\mathbb T^2}\chi^\gamma(y)dy=0
\end{equation*}
with
\begin{equation*}
B^{\alpha\beta}:=A^{\alpha\beta}+A^{\alpha\gamma}\partial_{y_\gamma}\chi^\beta+\partial_{y_\gamma}\bigl(A^{\gamma\alpha}\chi^\beta\bigr).
\end{equation*}
\end{enumerate}
\emph{We always assume that $\bar{u}^1=\bar{u}^2=0$.}

The first-order correction $u^1\left(\cdot,\frac{\cdot}{\varepsilon}\right)$ to $u^\varepsilon$ does not satisfy homogeneous Dirichlet boundary conditions on $\partial\Omega$. It is therefore responsible for a $O\left(\frac{1}{\sqrt{\varepsilon}}\right)$ term in the $H^1(\Omega)$ estimates involving $u^1$. In order to correct this, one introduces a boundary layer function $\vartheta^\varepsilon_{u,bl}$ solution of \eqref{sysoscbord} with $\varphi(x,y):=-u^1(x,y)=-\chi^\alpha(y)\partial_{x_\alpha}u^0(x)$. Note that $u^1\left(\cdot,\frac{\cdot}{\varepsilon}\right)+\vartheta^\varepsilon_{u,bl}$ belongs to $H^1_0(\Omega)$.

\subsection{Error estimates}
We extend here the estimates of Moskow and Vogelius (cf. \cite{moscovog} section $2$) to systems. 
\begin{prop}\label{theoenestH1}
Assume that $u^0\in H^2(\Omega)$.\\
Then
\begin{equation}\label{enestH1ineq}
\begin{Vmatrix}
u^\varepsilon(x)-u^0(x)-\varepsilon u^1\bigl(x,\frac{x}{\varepsilon}\bigr)-\varepsilon\vartheta_{u,bl}^\varepsilon(x)
\end{Vmatrix}_{H^1(\Omega)}\leq C\varepsilon
\begin{Vmatrix}
u^0
\end{Vmatrix}_{H^2(\Omega)}
\end{equation}
for $C>0$ independent of $\varepsilon$ and $u^0$.
\end{prop}

The proof relies on energy estimates on the error 
\begin{equation*}
e^\varepsilon:=u^\varepsilon(x)-u^0(x)-\varepsilon u^1\bigl(x,\frac{x}{\varepsilon}\bigr)-\varepsilon\vartheta_{u,bl}^\varepsilon(x).
\end{equation*}
It is a solution of the following system
\begin{equation}\label{syseeps}
\left\{
\begin{array}{rll}
-\nabla \cdot A\bigl(\frac{x}{\varepsilon}\bigr)\nabla e^\varepsilon=&r^\varepsilon,& x\in\Omega\\
e^\varepsilon=&0,& x\in\partial\Omega
\end{array}
\right.
\end{equation}
where 
\begin{equation}\label{exprreps}
r^\varepsilon:=f+\nabla\cdot\biggl[A\Bigl(\frac{x}{\varepsilon}\Bigr)\nabla u^0\biggr]+\varepsilon\nabla\cdot\biggl[A\Bigl(\frac{x}{\varepsilon}\Bigr)\nabla u^1\Bigl(x,\frac{x}{\varepsilon}\Bigr)\biggr].
\end{equation}
We intend to prove that the $H^1(\Omega)$ norm of $e^\varepsilon$ is of order $\varepsilon$ by showing that the source term $r^\varepsilon$ in \eqref{syseeps} is of order $\varepsilon$ in $H^{-1}(\Omega)$. It is not clear, looking at \eqref{exprreps}, that the latter is true. To face this issue, we invoke the classic key lemma, which can be proven using Fourier series expansions:
\begin{lem}\label{lemmedivrotd2}
Let $v=\begin{pmatrix}v_1\\v_2\end{pmatrix}\in C^\infty(\mathbb T^2;\mathbb R^2)$.\\
Assume 
\begin{equation*}
\nabla\cdot v=0\qquad \text{and}\qquad \int_{\mathbb T^2}v=0.
\end{equation*}
Then there exists $\psi=\psi(y)\in\mathbb R$ such that $v=\nabla^\perp\psi=\begin{pmatrix}-\partial_2\psi\\ \partial_1\psi\end{pmatrix}$.
\end{lem}

\begin{proof}[Proof of proposition \ref{theoenestH1}]
Expanding the source term $r^\varepsilon$ yields 
\begin{align}\label{repsdevpt}
r^\varepsilon=&\frac{1}{\varepsilon}\biggl[\Bigl[\nabla_y\cdot A(y)\nabla_yu^1\Bigr]\Bigl(x,\frac{x}{\varepsilon}\Bigr)+\Bigl[\nabla_y\cdot A(y)\nabla_xu^0\Bigr]\Bigl(x,\frac{x}{\varepsilon}\Bigr)\biggr]\\
&\:+f+\Bigl[\nabla_x\cdot A(y)\nabla_xu^0\Bigr]\Bigl(x,\frac{x}{\varepsilon}\Bigr)+\Bigl[\nabla_x\cdot A(y)\nabla_yu^1\Bigr]\Bigl(x,\frac{x}{\varepsilon}\Bigr)+\Bigl[\nabla_y\cdot A(y)\nabla_xu^1\Bigr]\Bigl(x,\frac{x}{\varepsilon}\Bigr)\nonumber\\
&\;\qquad+\varepsilon\Bigl[\nabla_x\cdot A(y)\nabla_xu^1\Bigr]\Bigl(x,\frac{x}{\varepsilon}\Bigr)\nonumber.
\end{align}
The leading idea is to get rid of terms of order $0$ and $-1$ in $\varepsilon$. We call
\begin{equation*}
v:=A(y)\nabla_yu^1+A(y)\nabla_xu^0
\end{equation*}
and notice that
\begin{align}
\nabla_y\cdot v&=\nabla_y\cdot A(y)\nabla_xu^0+\nabla_y\cdot A(y)\nabla_yu^1(x,y)\nonumber\\
&=\partial_{y_\alpha}A^{\alpha\beta}(y)\partial_{x_\beta}u^0+\partial{y_\alpha}\bigl(A^{\alpha\beta}(y)\partial_{y_\beta}\chi^\gamma(y)\bigr)\partial_{x_\gamma}u^0\nonumber\\
&=0\label{vdivy0}
\end{align}
because $\chi^\gamma$ solves \eqref{eqdefchi}.
Thus the $\varepsilon^{-1}$ order term in \eqref{repsdevpt} cancels and it remains to handle the zeroth order term:
\begin{multline}\label{zerothorderreps}
f+\Bigl[\nabla_x\cdot A(y)\nabla_xu^0\Bigr]\Bigl(x,\frac{x}{\varepsilon}\Bigr)+\Bigl[\nabla_x\cdot A(y)\nabla_yu^1\Bigr]\Bigl(x,\frac{x}{\varepsilon}\Bigr)+\Bigl[\nabla_y\cdot A(y)\nabla_xu^1\Bigr]\Bigl(x,\frac{x}{\varepsilon}\Bigr)\\
=f+\Bigl[\nabla_x\cdot v\Bigr]\Bigl(x,\frac{x}{\varepsilon}\Bigr)+\Bigl[\nabla_y\cdot A(y)\nabla_xu^1\Bigr]\Bigl(x,\frac{x}{\varepsilon}\Bigr).
\end{multline}
Here again, we take advantage of \eqref{vdivy0} and the definition of $A^0$: on the one hand
\begin{equation*}
\nabla_y\cdot\bigl(v-A^0\nabla u^0\bigr)=0
\end{equation*}
and on the other hand
\begin{equation*}
\int_{\mathbb T^2}\bigl(v-A^0\nabla u^0\bigr)=0.
\end{equation*}
It follows that one can apply lemma \ref{lemmedivrotd2}, component by component, and get a function $\psi=\psi(x,y)\in\mathbb R^N$ such that
\begin{equation*}
v-A^0\nabla u^0=\nabla^\perp_y\psi.
\end{equation*}
Due to the fact that $v-A^0\nabla u^0$ is a function of separated variables $x$ and $y$, 
$\psi$ itself is and factors into 
\begin{equation}\label{sepvarpsi}
\psi(x,y)=\Psi(y)\nabla u^0(x).
\end{equation}
The function $\Psi=\bigl(\Psi^\alpha(y)\bigr)_{1\leq\alpha\leq 2}\in {M_N(\mathbb R)}^2$ is given by the lemma and is therefore of class $C^\infty$. As $u^0$ is assumed to be in $H^2(\Omega)$, $\psi$ is in $H^1(\Omega)$ with respect to $x$. We set 
\begin{equation*}
w:=\nabla_x^\perp\psi
\end{equation*}
of regularity $L^2(\Omega)$ towards $x$, we compute
\begin{equation*}
\nabla_y\cdot w=\nabla_y\cdot\nabla_x^\perp\psi=-\nabla_x\cdot\nabla_y^\perp\psi=-\nabla_x\cdot v -f
\end{equation*}
and use this equality to simplify \eqref{zerothorderreps}
\begin{multline}
f+\Bigl[\nabla_x\cdot v\Bigr]\Bigl(x,\frac{x}{\varepsilon}\Bigr)+\Bigl[\nabla_y\cdot A(y)\nabla_xu^1\Bigr]\Bigl(x,\frac{x}{\varepsilon}\Bigr)=-\Bigl[\nabla_y\cdot w\Bigr]\Bigl(x,\frac{x}{\varepsilon}\Bigr)\\
+\Bigl[\nabla_y\cdot A(y)\nabla_xu^1\Bigr]\Bigl(x,\frac{x}{\varepsilon}\Bigr).
\end{multline}
Finally, using $\nabla_x\cdot w=0$ one obtains
\begin{align}\label{exprrepsmodif}
r^\varepsilon&=-\Bigl[\nabla_y\cdot w\Bigr]\Bigl(x,\frac{x}{\varepsilon}\Bigr)+\varepsilon\nabla\cdot\biggl[A\Bigl(\frac{x}{\varepsilon}\Bigr)\nabla_xu^1\Bigl(x,\frac{x}{\varepsilon}\Bigr)\biggr]\nonumber\\
&=-\Bigl[\nabla_y\cdot w\Bigr]\Bigl(x,\frac{x}{\varepsilon}\Bigr)-\varepsilon\Bigl[\nabla_x\cdot w\Bigr]\Bigl(x,\frac{x}{\varepsilon}\Bigr)+\varepsilon\nabla\cdot\biggl[A\Bigl(\frac{x}{\varepsilon}\Bigr)\nabla_xu^1\biggr]\Bigl(x,\frac{x}{\varepsilon}\Bigr)\nonumber\\
&=-\varepsilon\nabla\cdot\biggl[w\Bigl(x,\frac{x}{\varepsilon}\Bigr)\biggr]+\varepsilon\nabla\cdot\biggl[A\Bigl(\frac{x}{\varepsilon}\Bigr)\nabla_xu^1\Bigl(x,\frac{x}{\varepsilon}\Bigr)\biggr].
\end{align}

It remains to estimate the $H^{-1}(\Omega)$ norm of $r^\varepsilon$. The expression \eqref{exprrepsmodif} is convenient for two reasons: it is a sum of two terms of order $1$ in $\varepsilon$ and is written in divergence form. Moreover, $w\bigl(\cdot,\frac{\cdot}{\varepsilon}\bigr)$ as well as $A\bigl(\frac{\cdot}{\varepsilon}\bigr)\nabla_xu^1\bigl(\cdot,\frac{\cdot}{\varepsilon}\bigr)$ belong to $L^2(\Omega)$ and we have
\begin{align*}
\begin{Vmatrix}
w\bigl(\cdot,\frac{\cdot}{\varepsilon}\bigr)
\end{Vmatrix}_{L^2(\Omega)}&\leq C
\begin{Vmatrix}
u^0
\end{Vmatrix}_{H^2(\Omega)}\\
\begin{Vmatrix}
A\bigl(\frac{\cdot}{\varepsilon}\bigr)\nabla_xu^1\bigl(\cdot,\frac{\cdot}{\varepsilon}\bigr)
\end{Vmatrix}_{L^2(\Omega)}&\leq C
\begin{Vmatrix}
u^0
\end{Vmatrix}_{H^2(\Omega)}.
\end{align*}
Consequently, for all $\phi\in H^1_0(\Omega)$, by 
Cauchy-Schwarz inequality
\begin{align*}
&\left|\left\langle r^\varepsilon\Bigl(x,\frac{x}{\varepsilon}\Bigr),\phi(x)\right\rangle_{H^{-1}(\Omega),H^1_0(\Omega)}\right|\\
&\qquad\qquad=\left|\left\langle-\varepsilon\nabla\cdot\biggl[w\Bigl(x,\frac{x}{\varepsilon}\Bigr)\biggr]+\varepsilon\nabla\cdot\biggl[A\Bigl(\frac{x}{\varepsilon}\Bigr)\nabla_xu^1\Bigl(x,\frac{x}{\varepsilon}\Bigr)\biggr],\phi(x)\right\rangle_{H^{-1}(\Omega),H^1_0(\Omega)}\right|\\
&\qquad\qquad=\left|\varepsilon\int_{\Omega}w\Bigl(x,\frac{x}{\varepsilon}\Bigr)\cdot\nabla\phi(x)dx-\varepsilon\int_{\Omega}A\Bigl(\frac{x}{\varepsilon}\Bigr)\nabla_xu^1\Bigl(x,\frac{x}{\varepsilon}\Bigr)\cdot\nabla\phi(x)dx\right|\\
&\qquad\qquad\leq C\varepsilon
\begin{Vmatrix}
u^0
\end{Vmatrix}_{H^2(\Omega)}
\begin{Vmatrix}
\phi
\end{Vmatrix}_{H^{1}_0(\Omega)}
\end{align*}
which concludes the proof.
\end{proof}

\begin{cor}\label{coruepsu0L2}
Assume that $u^0\in H^2(\Omega)$.\\
Then 
\begin{equation}\label{inequepsu0L2}
\begin{Vmatrix}
u^\varepsilon(x)-u^0(x)
\end{Vmatrix}_{L^2(\Omega)}\leq C\varepsilon^\frac{1}{2}
\begin{Vmatrix}
u^0
\end{Vmatrix}_{H^2(\Omega)}.
\end{equation}
\end{cor}

\begin{proof}
By the triangular inequality, we get
\begin{multline*}
\begin{Vmatrix}
u^\varepsilon(x)-u^0(x)
\end{Vmatrix}_{L^2(\Omega)}\\
\leq
\begin{Vmatrix}
u^\varepsilon(x)-u^0(x)-\varepsilon u^1\bigl(x,\frac{x}{\varepsilon}\bigr)-\varepsilon\vartheta_{u,bl}^\varepsilon(x)
\end{Vmatrix}_{H^1(\Omega)}
+\varepsilon
\begin{Vmatrix}
u^1\bigl(x,\frac{x}{\varepsilon}\bigr)
\end{Vmatrix}_{L^2(\Omega)}
+\varepsilon
\begin{Vmatrix}
\vartheta_{u,bl}^\varepsilon(x)
\end{Vmatrix}_{L^2(\Omega)}.
\end{multline*}
The estimate \eqref{inequepsu0L2} now follows from \eqref{enestH1ineq}, the uniform boundedness of $u^1\bigl(\cdot,\frac{\cdot}{\varepsilon}\bigr)$ in $L^2(\Omega)$ and from the bound 
\begin{equation}\label{badestthetablepsmajH1/2}
\begin{Vmatrix}
\vartheta_{u,bl}^{\varepsilon}
\end{Vmatrix}_{L^2(\Omega)}
\leq
\begin{Vmatrix}
\vartheta_{u,bl}^{\varepsilon}
\end{Vmatrix}_{H^1(\Omega)}
\leq C
\begin{Vmatrix}
\chi^\alpha\bigl(\frac{x}{\varepsilon}\bigr)\partial_{x_\alpha}u^0(x)
\end{Vmatrix}_{H^\frac{1}{2}(\partial\Omega)}
\leq C\varepsilon^{-\frac{1}{2}}
\begin{Vmatrix}
u^0
\end{Vmatrix}_{H^2(\Omega)}.\qedhere
\end{equation}
\end{proof}

From corollary \ref{coruepsu0L2}, one easily gets a similar $L^2(\Omega)$ estimate under weaker assumptions on $u^0$.

\begin{cor}\label{corestuepsu0H1omega}
Assume that $u^0\in H^{1+\omega}(\Omega)$, with $0\leq\omega\leq 1$.\\
Then 
\begin{equation}\label{inequepsu0L2bis}
\begin{Vmatrix}
u^\varepsilon(x)-u^0(x)
\end{Vmatrix}_{L^2(\Omega)}\leq C\varepsilon^\frac{\omega}{2}
\begin{Vmatrix}
u^0
\end{Vmatrix}_{H^{1+\omega}(\Omega)}.
\end{equation}
\end{cor}

\begin{proof}
A straightforward energy estimate on the elliptic system satisfied by $u^\varepsilon-u^0$ yields
\begin{equation}\label{inequepsu0L2ter}
\begin{Vmatrix}
u^\varepsilon(x)-u^0(x)
\end{Vmatrix}_{L^2(\Omega)}
\leq C
\begin{Vmatrix}
u^\varepsilon(x)-u^0(x)
\end{Vmatrix}_{H^1_0(\Omega)}\leq C
\begin{Vmatrix}
u^0
\end{Vmatrix}_{H^{1}(\Omega)}.
\end{equation}
Inequality \eqref{inequepsu0L2bis} now comes from interpolating \eqref{inequepsu0L2ter} and \eqref{inequepsu0L2}. The main idea is to think of \eqref{inequepsu0L2ter} (resp. \eqref{inequepsu0L2}) as the statement that the linear operator 
\begin{equation*}
u^0\longmapsto u^\varepsilon(x)-u^0(x)
\end{equation*}
is bounded from $H^1(\Omega)$ to $L^2(\Omega)$ (resp. from $H^2(\Omega)$ to $L^2(\Omega)$) and then to interpolate.
\end{proof}

We call now $\vartheta^{2,\varepsilon}_{u,bl}$ the solution of \eqref{sysoscbord} with $\varphi(x,y):=-u^2(x,y)$, whose introduction is motivated by the same reasons as $\vartheta^\varepsilon_{u,bl}$, 
and state the proposition:

\begin{prop}\label{theoenestL2}
Assume that $u^0\in H^3(\Omega)$.\\
Then
\begin{equation}\label{enestL2ineq}
\begin{Vmatrix}
u^\varepsilon(x)-u^0(x)-\varepsilon u^1\bigl(x,\frac{x}{\varepsilon}\bigr)-\varepsilon\vartheta_{u,bl}^\varepsilon(x)
\end{Vmatrix}_{L^2(\Omega)}\leq C\varepsilon^\frac{3}{2}
\begin{Vmatrix}
u^0
\end{Vmatrix}_{H^3(\Omega)}.
\end{equation}
\end{prop}

\begin{proof}
The proof of \eqref{enestL2ineq} relies on a global energy estimate found in \cite{dgvnm}, section $3.3$:
\begin{equation}\label{globalenestdgvnm}
\begin{Vmatrix}
u^\varepsilon(x)-u^0(x)-\varepsilon u^1\bigl(x,\frac{x}{\varepsilon}\bigr)-\varepsilon\vartheta_{u,bl}^\varepsilon(x)-\varepsilon^2u^2\bigl(x,\frac{x}{\varepsilon}\bigr)-\varepsilon^2\vartheta_{u,bl}^{2,\varepsilon}(x)
\end{Vmatrix}_{H^1(\Omega)}=O\bigl(\varepsilon^2\bigr)
\end{equation}
It requires $u^0\in H^3(\Omega)$ and it can be showned using the same ideas than those involved in \eqref{enestH1ineq}, the key being again lemma \ref{lemmedivrotd2}. Following the lines of \cite{dgvnm} it becomes clear that the precised estimate
\begin{equation*}
\begin{Vmatrix}
u^\varepsilon(x)-u^0(x)-\varepsilon u^1\bigl(x,\frac{x}{\varepsilon}\bigr)-\varepsilon\vartheta_{u,bl}^\varepsilon(x)-\varepsilon^2u^2\bigl(x,\frac{x}{\varepsilon}\bigr)-\varepsilon^2\vartheta_{u,bl}^{2,\varepsilon}(x)
\end{Vmatrix}_{H^1(\Omega)}\leq C\varepsilon^2
\begin{Vmatrix}
u^0
\end{Vmatrix}_{H^3(\Omega)}
\end{equation*}
holds. It is nothing but a consequence of the way the involved functions factor in the product of a function depending only on $y$ and of $\nabla u^0$ (cf. \eqref{sepvarpsi}). We conclude, applying the triangular inequality, that
\begin{multline*}
\begin{Vmatrix}
u^\varepsilon(x)-u^0(x)-\varepsilon u^1\bigl(x,\frac{x}{\varepsilon}\bigr)-\varepsilon\vartheta_{u,bl}^\varepsilon(x)
\end{Vmatrix}_{L^2(\Omega)}\\
\leq
\begin{Vmatrix}
u^\varepsilon(x)-u^0(x)-\varepsilon u^1\bigl(x,\frac{x}{\varepsilon}\bigr)-\varepsilon\vartheta_{u,bl}^\varepsilon(x)-\varepsilon^2u^2\bigl(x,\frac{x}{\varepsilon}\bigr)-\varepsilon^2\vartheta_{u,bl}^{2,\varepsilon}(x)
\end{Vmatrix}_{H^1(\Omega)}
\\+\varepsilon^2
\begin{Vmatrix}
u^2\bigl(x,\frac{x}{\varepsilon}\bigr)
\end{Vmatrix}_{L^2(\Omega)}
+\varepsilon^2
\begin{Vmatrix}
\vartheta_{u,bl}^{2,\varepsilon}(x)
\end{Vmatrix}_{L^2(\Omega)}.
\end{multline*}
The estimate \eqref{enestL2ineq} now follows from \eqref{globalenestdgvnm}, the uniform boundedness of $u^2\bigl(\cdot,\frac{\cdot}{\varepsilon}\bigr)$ in $L^2(\Omega)$ and the \eqref{badestthetablepsmajH1/2}-like bound 
\begin{equation*}
\begin{Vmatrix}
\vartheta_{u,bl}^{2,\varepsilon}
\end{Vmatrix}_{H^1(\Omega)}\leq
C\varepsilon^{-\frac{1}{2}}\begin{Vmatrix}
u^0
\end{Vmatrix}_{H^3(\Omega)}. \qedhere
\end{equation*}
\end{proof}

We conclude this section focusing on a \eqref{enestL2ineq}-like estimate for $u^0$ satisfying a weaker assumption. 
\begin{theo}\label{theoL2ineqomega}
Assume $u^0\in H^{2+\omega}(\Omega)$, with $0\leq\omega\leq 1$.\\
Then
\begin{equation}\label{enestL2ineqomega}
\begin{Vmatrix}
u^\varepsilon(x)-u^0(x)-\varepsilon u^1\bigl(x,\frac{x}{\varepsilon}\bigr)-\varepsilon\vartheta_{u,bl}^\varepsilon(x)
\end{Vmatrix}_{L^2(\Omega)}\leq C\varepsilon^{1+\frac{\omega}{2}}
\begin{Vmatrix}
u^0
\end{Vmatrix}_{H^{2+\omega}(\Omega)}.
\end{equation}
\end{theo}

As in the proof of corollary \ref{corestuepsu0H1omega}, estimate \eqref{enestH1ineq} (resp. \eqref{enestL2ineq}) states that the linear operator 
\begin{equation*}
u^0\longmapsto u^\varepsilon(x)-u^0(x)-\varepsilon u^1\bigl(x,\frac{x}{\varepsilon}\bigr)-\varepsilon\vartheta_{u,bl}^\varepsilon(x)
\end{equation*}
is bounded from $H^2(\Omega)$ to $L^2(\Omega)$ (resp. from $H^3(\Omega)$ to $L^2(\Omega)$). By interpolating between the two linear operators, one gets \eqref{enestL2ineqomega}. All details can be found in \cite{moscovog} and apply without any change to the case $N>1$.

\selectlanguage{english}

\section{Homogenization of boundary layer type systems
}\label{sechomblsys}
Throughout this section we are interested in the homogenization of the boundary layer type system
\begin{equation}\label{sysvartheta}
\left\{
\begin{array}{rll}
-\nabla \cdot A\bigl(\frac{x}{\varepsilon}\bigr)\nabla \vartheta_{u,bl}^\varepsilon=&0,& x\in \Omega\\
\vartheta_{u,bl}^\varepsilon=&-\chi^\alpha\bigl(\frac{x}{\varepsilon}\bigr)\partial_{x_\alpha}u^0(x),& x\in \partial \Omega
\end{array}
\right.
\end{equation}
that is in the study of the asymptotic behaviour of the sequence $\vartheta_{u,bl}^\varepsilon$ when $\varepsilon$ tends to $0$. This means we both look for a possible limit of the sequence and for estimates in norm of the speed of convergence. 
For all this section, we assume that $u^0$ solves \eqref{sysu0}, with $f\in L^2(\Omega)$ and $\varphi_0=0$.

This is a crucial step in the proof of theorems \ref{theoasylisse} and \ref{theoasypol}. As explained in the introduction, there is no regularity issue when $\Omega$ is smooth. On the contrary, when $\Omega$ is a polygon, we concentrate on minimal regularity. That is why we give two convergence rates: the first under the minimal assumption $u^0\in H^{2+\omega}(\Omega)$ with $0<\omega<1$, the second under the stronger regularity assumption $u^0\in H^3(\Omega)\cap C^2(\overline{\Omega})$, where we focus on improving the speed of convergence.

\emph{For notational convenience, let us write in this section $\vartheta_{bl}^\varepsilon$ instead of $\vartheta_{u,bl}^\varepsilon$.}

\subsection{Smooth uniformly convex domains}
Assume that $\Omega\subset\mathbb R^2$ is a smooth (say $C^\infty$) uniformly convex domain i.e. all principal curvatures are bounded from below; see \cite{bre} section III.$7$ for another definition. The regularizing properties of elliptic operators in smooth domains yield that $u^0\in C^\infty(\overline{\Omega})$ (see \cite{adn2} theorem $10.5$). Therefore, the boundary data function $\varphi(x,y)=-u^1(x,y)=-\chi^\alpha(y)\partial_{x_\alpha}u^0(x)$ is a smooth function. Note that we do not need assumption {\bf (A4)}, i.e. the symmetry of $A$. 
\begin{theo}[G\'erard-Varet and Masmoudi in \cite{dgvnm2}]\label{theodgvnm2}
For all $1\leq p<\infty$ there exists $\varphi^*\in L^p(\partial\Omega)$ such that $\vartheta_{bl}^{\varepsilon}$ converges in $L^2(\Omega)$ towards $\vartheta_{bl}^*\in L^p(\Omega)$ solution of
\begin{equation*}
\left\{
\begin{array}{rll}
-\nabla \cdot A^0\nabla \vartheta_{bl}^*=&0,& x\in \Omega\\
\vartheta_{bl}^*=&\varphi^*(x),& x\in \partial\Omega
\end{array}
\right. .
\end{equation*}
Moreover, for all $0\leq\gamma<\frac{1}{11}$,
\begin{equation}\label{estCVordre0luc}
\begin{Vmatrix}
\vartheta_{bl}^{\varepsilon}-\vartheta_{bl}^*
\end{Vmatrix}_{L^2(\Omega)}=O\bigl(\varepsilon^{\gamma}\bigr).
\end{equation}
\end{theo}

We do not attempt to weaken the regularity assumption on $\Omega$, which itself implies strong regularity on $u^0$. 
For details concerning the proof and relevant remarks, we refer to \cite{dgvnm2}.

\subsection{Convex polygonal domains}
\label{convexpol}
Let us assume $\Omega$ to be a bounded convex polygonal domain with $M$ edges, supported by the lines $K^k$ of unitary inward normal $n^k\in S^{1}$. Thus
\begin{equation*}
\Omega=\bigcap_{k=1}^M\bigl\{x,\ n^k\cdot x>c^k\bigr\} 
\end{equation*}
with $c^k\in\mathbb R$, and for all $1\leq k\leq M$,
\begin{equation*}
K^k=\bigl\{x,\ n^k\cdot x=c^k\bigr\}.
\end{equation*}

Beyond this first assumption on $\Omega$ we require either
\begin{description}
\item[(RAT) rationality] for all $1\leq k\leq M$, 
\begin{equation}\label{crat}
n^k\in\mathbb R\mathbb Q^2
\end{equation}
or
\item[(DIV) small divisors] there exists $C,\ l>0$ such that for all $1\leq k\leq M$,
\begin{equation}\label{ptdiva}
\forall\xi=(\xi_1,\xi_2)\in\mathbb Z^2\setminus\{0\},\qquad |P_{{n^k}^\perp}(\xi)|\geq C|\xi|^{-l}
\end{equation}
where $P_{{n^k}^\perp}$ is the orthogonal projector on ${n^k}^\perp$. 
\end{description}
As $\Omega\subset\mathbb R^2$, condition \eqref{ptdiva} boils down to
\begin{equation}\label{ptdivb}
\forall\xi\in\mathbb Z^2\setminus\{0\},\qquad |n^k\cdot\xi|\geq C|\xi|^{-l}
\end{equation}
where $n^k\cdot\xi:=n^k_1\xi_1+n^k_2\xi_2$.
Note that a vector $n\in\mathbb R^2$ cannot satisfy both \eqref{crat} and \eqref{ptdiva} or \eqref{ptdivb}. 

Keeping in mind that $\vartheta_{bl}^\varepsilon$ solves \eqref{sysvartheta}, one has the following convergence theorems. Note that as soon as we invoke the regularity theorem \ref{theoregdaugekmr} in the proofs, we need the symmetry assumption {\bf (A4)} on $A$.
\begin{theo}\label{CVhomorat}
Assume $\Omega$ satisfies {\bf (RAT)}. Assume furthermore that $u^0\in H^{2+\omega}(\Omega)$ with $0<\omega<1$ (resp. $u^0\in H^3(\Omega)\cap C^2(\overline{\Omega})$).\\
Then there exists a sequence $(\varepsilon_n)$ and $(V^{k,\alpha,*})_{
\begin{subarray}{c}
1\leq k\leq M\\
1\leq\alpha\leq 2
\end{subarray}}\in {M_N(\mathbb R)}^{2\times M}$ such that $\vartheta_{bl}^{\varepsilon_n}$ solution of \eqref{sysoscbord} converges in $L^2(\Omega)$ towards $\vartheta_{bl}^*$ solution of
\begin{equation}\label{sysoscbordhom}
\left\{
\begin{array}{rll}
-\nabla \cdot A^0\nabla \vartheta_{bl}^*=&0,& x\in \Omega\\
\vartheta_{bl}^*=&-V^{k,\alpha,*}\partial_{x_\alpha}u^0(x),& x\in \partial\Omega\cap K^k, \text{ for all } 1\leq k\leq M
\end{array}
\right. .
\end{equation}
Moreover, we have the following convergence rates:
\begin{enumerate}
\item if $u^0\in H^{2+\omega}(\Omega)$, then there exists $0<\gamma<\frac{\omega}{2}$ such that
\begin{equation*}
\begin{Vmatrix}
\vartheta_{bl}^{\varepsilon_n}-\vartheta_{bl}^*
\end{Vmatrix}_{L^2(\Omega)}=O\bigl(\varepsilon_n^{\gamma}\bigr);
\end{equation*}
\item if $u^0\in H^3(\Omega)\cap C^2(\overline{\Omega})$, then 
\begin{equation*}
\begin{Vmatrix}
\vartheta_{bl}^{\varepsilon_n}-\vartheta_{bl}^*
\end{Vmatrix}_{L^2(\Omega)}=O\bigl(\varepsilon_n^{\frac{1}{2}}\bigr).
\end{equation*}
\end{enumerate}
\end{theo}

\begin{theo}\label{CVhomodiv}
Assume $\Omega$ satisfies {\bf (DIV)}. Assume furthermore that $u^0\in H^{2+\omega}(\Omega)$ with $0<\omega<1$ (resp. $u^0\in H^3(\Omega)\cap C^2(\overline{\Omega})$).\\
Then there exists $(V^{k,\alpha,*})_{
\begin{subarray}{c}
1\leq k\leq M\\
1\leq\alpha\leq 2
\end{subarray}}\in {M_N(\mathbb R)}^{2\times M}$ such that $\vartheta_{bl}^{\varepsilon}$ solution of \eqref{sysoscbord} converges in $L^2(\Omega)$ towards $\vartheta_{bl}^*$ solution of
\begin{equation}\label{sysoscbordhomdiv}
\left\{
\begin{array}{rll}
-\nabla \cdot A^0\nabla \vartheta_{bl}^*=&0,& x\in \Omega\\
\vartheta_{bl}^*=&-V^{k,\alpha,*}\partial_{x_\alpha}u^0(x),& x\in \partial \Omega\cap K^k, \text{ for all } 1\leq k\leq M
\end{array}
\right. .
\end{equation}
Moreover, we have the following convergence rates:
\begin{enumerate}
\item if $u^0\in H^{2+\omega}(\Omega)$, then there exists $0<\gamma<\frac{\omega}{2}$ such that
\begin{equation}\label{estCVordre0divbis}
\begin{Vmatrix}
\vartheta_{bl}^{\varepsilon}-\vartheta_{bl}^*
\end{Vmatrix}_{L^2(\Omega)}=O\bigl(\varepsilon^{\gamma}\bigr);
\end{equation}
\item if $u^0\in H^3(\Omega)\cap C^2(\overline{\Omega})$, then
\begin{equation}\label{estCVordre0div}
\begin{Vmatrix}
\vartheta_{bl}^{\varepsilon}-\vartheta_{bl}^*
\end{Vmatrix}_{L^2(\Omega)}=O\bigl(\varepsilon^{\frac{1}{2}}\bigr).
\end{equation}
\end{enumerate}
\end{theo}

The only, but major, difference between theorem \ref{CVhomorat} and \ref{CVhomodiv} is that, in the small divisors case, convergence holds for the whole sequence, whereas in the rational case, convergence takes place up to the extraction of a subsequence $(\varepsilon_n)$, the constant matrices $V^{k,\alpha,*}$ depending on $(\varepsilon_n)$. 

\subsection{Proof of theorem \ref{CVhomodiv}}

The proofs of theorems \ref{CVhomorat} and \ref{CVhomodiv} follow the same steps. They differ mainly in one intermediate result, which explains why in the rational case, the convergence result is true only up to the extraction of a subsequence. Although we focus on the {\bf (DIV)} assumption, we underline the difference with assumption {\bf (RAT)}.

\subsubsection{Existence of the boundary layer tails}
Let us show the existence of the matrices $V^{k,\alpha,*}$. Let $1\leq k\leq M$ and $1\leq\alpha\leq 2$. We are interested in the boundary layer profile in the vicinity of vertex $k$. Thus, introduce $v_{bl}^{k,\alpha,\varepsilon}$ solution of
\begin{equation*}
\left\{
\begin{array}{rll}
-\nabla_y \cdot A(y)\nabla_y v_{bl}^{k,\alpha,\varepsilon}=&0,& y\in \Omega^{k,\varepsilon}\\
v_{bl}^{k,\alpha,\varepsilon}=&\chi^\alpha(y),& y\in \partial \Omega^{k,\varepsilon}
\end{array}
\right. 
\end{equation*}
where $\Omega^{k,\varepsilon}:=\bigl\{y,\ n^k\cdot y-\frac{c^k}{\varepsilon}>0\bigr\}$.
Let $M^k\in M_2(\mathbb R)$ be an orthogonal matrix, mapping $e_2:=\begin{pmatrix}0\\1\end{pmatrix}$ to $n^k$. 

The following theorem describes the profile of $V^{k,\alpha,\varepsilon}:=v^{k,\alpha,\varepsilon}_{bl}(M^k\cdot)$.
\begin{theo}[G\'erard-Varet and Masmoudi in \cite{dgvnm}]\label{theodecbldiv}
Assume that $\Omega$ satisfies {\bf (DIV)}.\\
Then, 
\begin{enumerate}
\item for all $\varepsilon>0$, there exists $V^{k,\alpha,\varepsilon}\in C^\infty\Bigl(\mathbb T\times \bigl]\frac{c^k}{\varepsilon},\infty\bigr[\Bigr)$;
\item there exists a matrix $V^{k,\alpha,*}\in M_N(\mathbb R)$ such that for all $\beta\in\mathbb N^2$, for all $m\in\mathbb N$, there is a constant $C_{|\beta|,m}>0$ satisfying for all $\varepsilon>0$ and $z_2>\frac{c^k}{\varepsilon}$,
\begin{equation}\label{inegdecVkalphaeps}
\Bigl(1+\begin{vmatrix}z_2-\frac{c^k}{\varepsilon}\end{vmatrix}^m\Bigr)\sup_{z_1\in\mathbb R}
\begin{vmatrix}
\partial_{z}^\beta\bigl(V^{k,\alpha,\varepsilon}(z_1,z_2)-V^{k,\alpha,*}\bigr)
\end{vmatrix}\leq C_{|\beta|,m}.
\end{equation}
\end{enumerate}
\end{theo}

\begin{rem}
Note that \eqref{inegdecVkalphaeps} is true not only for $m\in\mathbb N$ but for $m\in\mathbb R$, $m>0$. Indeed, $[m]$ denoting the integer part of $m$, we have 
\begin{align*}
&\left|z_2-\frac{c^k}{\varepsilon}\right|^m<\left|z_2-\frac{c^k}{\varepsilon}\right|^{[m]+1}\mbox{, if } \left|z_2-\frac{c^k}{\varepsilon}\right|>1\\
&\left|z_2-\frac{c^k}{\varepsilon}\right|^m\leq 1 \mbox{, if } \left|z_2-\frac{c^k}{\varepsilon}\right|\leq 1.
\end{align*}
\end{rem}

\begin{rem}
If we rewrite the second statement of theorem \ref{theodecbldiv} in terms of $v_{bl}^{k,\alpha,\varepsilon}$ instead of $V^{k,\alpha,\varepsilon}$ we get:
for all $\beta\in\mathbb N^2$, for all $m\in\mathbb N$, there is a constant $C_{|\beta|,m}>0$ satisfying for all $\varepsilon>0$ and $y\in\Omega^{k,\varepsilon}$,
\begin{equation}\label{inegdecvblkalphaeps}
\Bigl(1+\begin{vmatrix}y\cdot n^k-\frac{c^k}{\varepsilon}\end{vmatrix}^m\Bigr)
\begin{vmatrix}
\partial_y^\beta\bigl(v_{bl}^{k,\alpha,\varepsilon}(y)-V^{k,\alpha,*}\bigr)
\end{vmatrix}\leq C_{|\beta|,m}.
\end{equation}
\end{rem}

\begin{rem}
If instead of {\bf (DIV)} we assume {\bf (RAT)}, the boundary layer tails $V^{k,\alpha,\varepsilon}$ still exist. Furthermore, an equivalent of theorem \ref{theodecbldiv} states: there exists a sequence $(\varepsilon_n)$ (here lies the main difference between the two assumptions), a constant matrix $V^{k,\alpha,*}\in M_N(\mathbb R)$ such that for all $m\in\mathbb N$, for all $z_2>\frac{c^k}{\varepsilon_n}$,
\begin{equation*}
\Bigl(1+\begin{vmatrix}z_2-\frac{c^k}{\varepsilon_n}\end{vmatrix}^m\Bigr)\sup_{z_1\in\mathbb R}
\begin{vmatrix}
\partial_{z}^\beta\bigl(V^{k,\alpha,\varepsilon_n}(z_1,z_2)-V^{k,\alpha,*}\bigr)
\end{vmatrix}\leq C_{|\beta|,m}.
\end{equation*}
The latter is sufficient to get our results. However, Moskow and Vogelius in \cite{moscovog}, as well as Allaire and Amar in \cite{allam} manage to prove an improved result under assumption {\bf (RAT)}: the convergence of the boundary layer towards its tail is exponential.
\end{rem}

\emph{Assume from now on that $\Omega$ satisfies {\bf (DIV)}. Let $0<\omega<1$ be fixed. The assumptions $u^0\in H^{2+\omega}(\Omega)$ with $0<\omega<1$ and $u^0\in H^3(\Omega)\cap C^2(\overline{\Omega})$ are treated in parallel.  
In both cases, by Sobolev injection, $u^0\in C^1(\overline{\Omega})$.}

\subsubsection{Well-posedness of \eqref{sysoscbordhomdiv}}
It is enough to prove that the boundary function of \eqref{sysoscbordhomdiv} belongs to $H^{\frac{1}{2}}(\partial\Omega)$. One constructs a lifting $\phi^*_{bl}$ of $\varphi^*_{bl}:=-V^{k,\alpha,*}\partial_{x_\alpha}u^0(x)$. There exists $G=\left(G^1,G^2\right)\in M_N(\mathbb R)\times M_N(\mathbb R)$ such that
\begin{equation}\label{liftingphi*}
\phi^*_{bl}=G^\alpha\partial_{x_\alpha} u^0.
\end{equation}
Following \cite{dgvnm}, one can show:
\begin{prop}\label{propboundarydatahomoH12}
If $u^0\in H^2(\Omega)$, then $\phi^*_{bl}\in H^1(\Omega)$. If $u^0$, in addition, belongs to $H^3(\Omega)$, then $\phi^*_{bl}$ belongs to $H^2(\Omega)$.
\end{prop}

\subsubsection{Sketch of the proof of the estimates \eqref{estCVordre0divbis} and \eqref{estCVordre0div}}
Our strategy is to split the problem of estimating $\vartheta_{bl}^{\varepsilon}-\vartheta_{bl}^*$ into three easier ones. 
For this purpose, we introduce $\vartheta_{bl}^{\varepsilon,*}$ solution of
\begin{equation}\label{sysoscaux}
\left\{
\begin{array}{rll}
-\nabla \cdot A\bigl(\frac{x}{\varepsilon}\bigr)\nabla \vartheta_{bl}^{\varepsilon,*}=&0,& x\in \Omega\\
\vartheta_{bl}^{\varepsilon,*}=&-V^{k,\alpha,*}\partial_{x_\alpha}u^0(x),& x\in \partial\Omega\cap K^k, \text{ for all } 1\leq k\leq M
\end{array}
\right. 
\end{equation}
to get via the triangular inequality:
\begin{equation*}
\begin{Vmatrix}
\vartheta_{bl}^\varepsilon-\vartheta_{bl}^*
\end{Vmatrix}_{L^2(\Omega)}
\leq
\begin{Vmatrix}
\vartheta_{bl}^{\varepsilon,*}-\vartheta_{bl}^*
\end{Vmatrix}_{L^2(\Omega)}
+
\begin{Vmatrix}
\vartheta_{bl}^\varepsilon-\vartheta_{bl}^{\varepsilon,*}
\end{Vmatrix}_{L^2(\Omega)}.
\end{equation*}
Note that $\vartheta_{bl}^{\varepsilon,*}$ is well defined because of proposition \ref{propboundarydatahomoH12}. 

The study of the first term seems to be more classic as the boundary data function of \eqref{sysoscbordhomdiv} and \eqref{sysoscaux} is not oscillating. The second term, on the contrary, requires a deep knowledge about the homogenization of boundary layer systems.

In fact $\vartheta_{bl}^\varepsilon-\vartheta_{bl}^{\varepsilon,*}$ is the solution of \eqref{sysoscbord} with $\varphi(x,y)=-\bigl(\chi^\alpha(y)-V^{k,\alpha,*}\bigr)\partial_{x_\alpha}u^0(x)$, for all $x\in\partial\Omega\cap K^k$, for all $y\in\mathbb R^2$; we call $u^{1,\varepsilon}_{bl}$ the difference $\vartheta_{bl}^\varepsilon-\vartheta_{bl}^{\varepsilon,*}$. It comes from proposition \ref{propboundarydatahomoH12} that $\varphi$ defined like this is in $H^{\frac{1}{2}}(\partial\Omega)$. Let $v_{bl}^{k,\varepsilon}:=-\bigl(v_{bl}^{k,\alpha,\varepsilon}-V^{k,\alpha,*}\bigr)\partial_{x_\alpha}u^0$. 
We expect $u^{1,\varepsilon}_{bl}$ to be close to $\sum_{k=1}^Mv_{bl}^{k,\varepsilon}\bigl(\cdot,\frac{\cdot}{\varepsilon}\bigr)$:
\begin{equation*}
\begin{Vmatrix}
u^{1,\varepsilon}_{bl}
\end{Vmatrix}_{L^2(\Omega)}
\leq
\sum_{k=1}^M\begin{Vmatrix}
v^{k,\varepsilon}_{bl}\bigl(x,\frac{x}{\varepsilon}\bigr)
\end{Vmatrix}_{L^2(\Omega)}
+\begin{Vmatrix}
u^{1,\varepsilon}_{bl}(x)-\sum_{k=1}^Mv^{k,\varepsilon}_{bl}\bigl(x,\frac{x}{\varepsilon}\bigr)
\end{Vmatrix}_{L^2(\Omega)}.
\end{equation*}

The rest of the proof is thus devoted to estimate each of the terms in the r.h.s. of:
\begin{multline}\label{dectriangbis}
\begin{Vmatrix}
\vartheta_{bl}^\varepsilon-\vartheta_{bl}^*
\end{Vmatrix}_{L^2(\Omega)}
\leq
\begin{Vmatrix}
\vartheta_{bl}^{\varepsilon,*}-\vartheta_{bl}^*
\end{Vmatrix}_{L^2(\Omega)}
+\sum_{k=1}^M\begin{Vmatrix}
v^{k,\varepsilon}_{bl}\bigl(x,\frac{x}{\varepsilon}\bigr)
\end{Vmatrix}_{L^2(\Omega)}\\
+\begin{Vmatrix}
u^{1,\varepsilon}_{bl}(x)-\sum_{k=1}^Mv^{k,\varepsilon}_{bl}\bigl(x,\frac{x}{\varepsilon}\bigr)
\end{Vmatrix}_{L^2(\Omega)}.
\end{multline}

\subsubsection{First term in the r.h.s of \eqref{dectriangbis}}
We resort to corollary \ref{corestuepsu0H1omega} to estimate this term. 
In order to get some convergence rate, we need to have a little more regularity on $\vartheta_{bl}^{*}$ than $\vartheta_{bl}^{*}\in H^1(\Omega)$. According to proposition \ref{propboundarydatahomoH12}, the lifting $\phi^*_{bl}$ of the boundary data of \eqref{sysoscbordhomdiv} belongs to $H^{1+\omega}(\Omega)$ (resp. $H^2(\Omega)$), provided that $u^0$ belongs to $H^{2+\omega}(\Omega)$ (resp. $H^3(\Omega)$). 

Let us treat the two assumptions on $u^0$ separately. If $u^0\in H^{2+\omega}(\Omega)$, it follows from the first point of theorem \ref{theoregdaugekmr},
 that $\vartheta_{bl}^*$ has $H^{1+\gamma}(\Omega)$ regularity for all $\gamma$ such that $0\leq\gamma\leq\omega$ and $\gamma\neq\frac{1}{2}$. Therefore,
\begin{equation*}
\begin{Vmatrix}
\vartheta_{bl}^{\varepsilon,*}-\vartheta_{bl}^*
\end{Vmatrix}_{L^2(\Omega)}=O\bigl(\varepsilon^\frac{\gamma}{2}\bigr).
\end{equation*}
If $u^0\in H^3(\Omega)$, the second point of theorem \ref{theoregdaugekmr} yields that $\vartheta_{bl}^*\in H^2(\Omega)$. Applying corollary \ref{coruepsu0L2} implies
\begin{equation*}
\begin{Vmatrix}
\vartheta_{bl}^{\varepsilon,*}-\vartheta_{bl}^*
\end{Vmatrix}_{L^2(\Omega)}=O\bigl(\varepsilon^\frac{1}{2}\bigr).
\end{equation*}

\subsubsection{Second term in the r.h.s of \eqref{dectriangbis}}
By linearity of the equations, the boundary layer tail $V^{k,*}(x)$ of $v_{bl}^{k,\varepsilon}(x,\cdot)$ is equal to
\begin{equation*}
V^{k,*}(x)=-V^{k,\alpha,*}\partial_{x_\alpha}u^0(x)+V^{k,\alpha,*}\partial_{x_\alpha}u^0(x)=0.
\end{equation*}
We deduce from theorem \ref{theodecbldiv}: for all $m\in\mathbb N$, there is a constant $C_m>0$ such that for all $\varepsilon>0$, for all $x\in\Omega$,
\begin{equation}\label{eqmajvblkepsm}
\biggl(1+\frac{\begin{vmatrix}x\cdot n^k-c^k\end{vmatrix}^m}{\varepsilon^m}\biggr)
\begin{vmatrix}
v_{bl}^{k,\varepsilon}\bigl(x,\frac{x}{\varepsilon}\bigr)
\end{vmatrix}\leq C_{m}.
\end{equation}
The uniformity in $x$ comes from the fact $u^0\in C^1(\overline{\Omega})$ and from the boundedness of $\Omega$.

\begin{prop}\label{lemvkepsblL2}
For all $1\leq k\leq M$, $\begin{Vmatrix}
v^{k,\varepsilon}_{bl}\bigl(x,\frac{x}{\varepsilon}\bigr)
\end{Vmatrix}_{L^2(\Omega)}=O\bigl(\varepsilon^\frac{1}{2}\bigr)$.
\end{prop}


\begin{proof}[Proof]
Let $m\in\mathbb N$. From \eqref{eqmajvblkepsm} we get
\begin{align*}
\begin{Vmatrix}
v_{bl}^{k,\varepsilon}\bigl(x,\frac{x}{\varepsilon}\bigr)
\end{Vmatrix}_{L^2(\Omega)}^2
&\leq C\int_{\Omega}\frac{1}{\Bigl(1+\frac{|x\cdot n^k-c^k|^m}{\varepsilon^m}\Bigr)^2}dx\\
&\leq C\int_{\widetilde{\Omega}}\frac{1}{\Bigl(1+\frac{u_2^m}{\varepsilon^m}\Bigr)^2}du
\end{align*}
where $\widetilde{\Omega}:= ^t\!\!\!M^k\Omega-\begin{pmatrix}0\\ c^k\end{pmatrix}$.
Therefore, we have to focus on
\begin{equation*}
\int_{[0,\infty[}\frac{1}{\Bigl(1+\frac{u_2^m}{\varepsilon^m}\Bigr)^2}du_2=\varepsilon^{2m}\int_{[0,\infty[}\frac{1}{\bigl(\varepsilon^{m}+u_2^m\bigr)^2}du_2.
\end{equation*}
For $2m>1$ the integral is convergent and
\begin{equation*}
\int_{[0,\infty[}\frac{1}{\bigl(\varepsilon^{m}+u_2^m\bigr)^2}du_2\leq \int_{[0,\varepsilon]}\frac{1}{\varepsilon^{2m}}+\int_{[\varepsilon,\infty[}\frac{1}{u_2^{2m}}du_2=O\bigl(\varepsilon^{-2m+1}\bigr).
\end{equation*}
We immediately deduce that 
\begin{equation*}
\begin{Vmatrix}
v_{bl}^{k,\varepsilon}\bigl(x,\frac{x}{\varepsilon}\bigr)
\end{Vmatrix}_{L^2(\Omega)}^2=O(\varepsilon)
\end{equation*}
which yields the result.
\end{proof}

\begin{rem}
It is easy to adapt the proof of proposition \ref{lemvkepsblL2} to get: for all $1\leq k\leq M$, for all $1\leq p\leq\infty$, $\begin{Vmatrix}
v^{k,\varepsilon}_{bl}\bigl(x,\frac{x}{\varepsilon}\bigr)
\end{Vmatrix}_{L^p(\Omega)}=O\bigl(\varepsilon^\frac{1}{p}\bigr)$.
For $p=\infty$ it is \eqref{eqmajvblkepsm} with $m=0$; for $1\leq p<\infty$ the proof follows the lines of the case $p=2$, except that one has to replace $2$ by $p$. In the same manner, it is very straightforward to deduce from \eqref{inegdecvblkalphaeps} that for all $1\leq p\leq\infty$, for all $\beta\in \mathbb N^d$, for all $m\in\mathbb N$,
\begin{equation}\label{ineqvblkepsalphaLp}
\begin{Vmatrix}
\partial_{y}^\beta\Bigl(v^{k,\alpha,\varepsilon}_{bl}\bigl(\frac{x}{\varepsilon}\bigr)-V^{k,\alpha,*}\Bigr)
\frac{\left|x\cdot n^k-c^k\right|^m}{\varepsilon^m}
\end{Vmatrix}_{L^p(\Omega)}=O\bigl(\varepsilon^\frac{1}{p}\bigr).
\end{equation}
\end{rem}

\subsubsection{Third term in the r.h.s of \eqref{dectriangbis}}

We proceed as usual by carrying out energy estimates on the error
\begin{equation*}
e^\varepsilon_{bl}:=u^{1,\varepsilon}_{bl}(x)-\sum_{k=1}^Mv^{k,\varepsilon}_{bl}\Bigl(x,\frac{x}{\varepsilon}\Bigr).
\end{equation*}
It solves the system 
\begin{equation*}
\left\{
\begin{array}{rll}
-\nabla \cdot A\bigl(\frac{x}{\varepsilon}\bigr)\nabla e_{bl}^{\varepsilon}=&r_{bl}^\varepsilon,& x\in \Omega\\
e_{bl}^{\varepsilon}=&\varphi_{bl}^\varepsilon,& x\in \partial\Omega
\end{array}
\right. 
\end{equation*}
where the source term 
\begin{equation*}
r_{bl}^\varepsilon:=\sum_{k=1}^M\biggl\{\nabla\cdot\biggl(A\Bigl(\frac{x}{\varepsilon}\Bigr)\nabla_xv_{bl}^{k,\varepsilon}\Bigl(x,\frac{x}{\varepsilon}\Bigr)\biggr)+\frac{1}{\varepsilon}\Bigl[\nabla_x\cdot A(y)\nabla_yv_{bl}^{k,\varepsilon}\Bigr]\Bigl(x,\frac{x}{\varepsilon}\Bigr)\biggr\}
\end{equation*}
and the piecewise defined boundary function
\begin{equation*}
\left.\varphi_{bl}^\varepsilon\right|_{\partial\Omega\cap K^k}:=-\biggl(\chi^\alpha\Bigl(\frac{x}{\varepsilon}\Bigr)-V^{k,\alpha,*}\biggr)\partial_{x_\alpha}u^0(x)-\sum_{k'=1}^Mv^{k',\varepsilon}_{bl}\Bigl(x,\frac{x}{\varepsilon}\Bigr)=-\sum_{k'\neq k}^Mv^{k',\varepsilon}_{bl}\Bigl(x,\frac{x}{\varepsilon}\Bigr).
\end{equation*}
We estimate separately $r_{bl}^\varepsilon$ (cf. lemma \ref{lemrblepsH-1}) and $\varphi_{bl}^\varepsilon$ (cf. lemma \ref{lemphiblepsH1/2}).

\begin{lem}\label{lemrblepsH-1}
The source term $r_{bl}^\varepsilon$ is 
\begin{enumerate}
\item of order $O\bigl(\varepsilon^\gamma)$ in $H^{-1}(\Omega)$ for all $0<\gamma<\frac{\omega}{2}$ if $u^0\in H^{2+\omega}(\Omega)$;
\item of order $O\bigl(\varepsilon^\frac{1}{2}\bigr)$ in $H^{-1}(\Omega)$ if $u^0\in C^2(\overline{\Omega})$.
\end{enumerate}
\end{lem}

\begin{proof}[Proof]
Assume $u^0\in H^{2+\omega}(\Omega)$ (resp. $u^0\in C^2(\overline{\Omega}$)). Let $1\leq k\leq M$ be fixed and consider 
\begin{equation*}
r_{bl}^{\varepsilon,k}:=\nabla\cdot\biggl(A\Bigl(\frac{x}{\varepsilon}\Bigr)\nabla_xv_{bl}^{k,\varepsilon}\Bigl(x,\frac{x}{\varepsilon}\Bigr)\biggr)+\frac{1}{\varepsilon}\Bigl[\nabla_x\cdot A(y)\nabla_yv_{bl}^{k,\varepsilon}\Bigr]\Bigl(x,\frac{x}{\varepsilon}\Bigr).
\end{equation*}

We focus on the first term of $r_{bl}^{\varepsilon,k}$. 
For all $\phi\in H^1_0(\Omega)$, 
\begin{align*}
&\left|\left\langle\nabla\cdot\biggl(A\Bigl(\frac{x}{\varepsilon}\Bigr)\nabla_xv_{bl}^{k,\varepsilon}\Bigl(x,\frac{x}{\varepsilon}\Bigr)\biggr),\phi(x)\right\rangle_{H^{-1}(\Omega),H^1_0(\Omega)}\right|\\
&\qquad\qquad=\left|\int_{\Omega}\biggl(A\Bigl(\frac{x}{\varepsilon}\Bigr)\nabla_xv_{bl}^{k,\varepsilon}\Bigl(x,\frac{x}{\varepsilon}\Bigr)\biggr)\nabla\phi(x)dx\right|\\
&\qquad\qquad\leq \begin{Vmatrix}
A\bigl(\frac{x}{\varepsilon}\bigr)\nabla_xv_{bl}^{k,\varepsilon}\bigl(x,\frac{x}{\varepsilon}\bigr)
\end{Vmatrix}_{L^2(\Omega)}
\begin{Vmatrix}
\nabla\phi
\end{Vmatrix}_{L^2(\Omega)}\\
&\qquad\qquad\leq C\begin{Vmatrix}
\nabla_xv_{bl}^{k,\varepsilon}\bigl(x,\frac{x}{\varepsilon}\bigr)
\end{Vmatrix}_{L^2(\Omega)}
\begin{Vmatrix}
\phi
\end{Vmatrix}_{H^1_0(\Omega)}.
\end{align*}
At this point we need to estimate $\begin{Vmatrix}
\nabla_xv_{bl}^{k,\varepsilon}\bigl(x,\frac{x}{\varepsilon}\bigr)
\end{Vmatrix}_{L^2(\Omega)}$. As 
\begin{equation*}
\nabla_xv_{bl}^{k,\varepsilon}\bigl(x,\frac{x}{\varepsilon}\bigr)=\Bigl(v^{k,\beta,\varepsilon}_{bl}\bigl(\frac{x}{\varepsilon}\bigr)-V^{k,\beta,*}\Bigr)\partial_{x_\alpha}\partial_{x_\beta}u^0,
\end{equation*}
the idea is to bound the $L^2(\Omega)$ norm of this term using a H\"older inequality and \eqref{ineqvblkepsalphaLp}. Doing so, one has to pay attention to the regularity of $u^0$, and to carefully choose the $L^p(\Omega)$ spaces involved.

If $u^0\in C^2(\overline{\Omega})$, 
\begin{equation*}
\begin{Vmatrix}
\Bigl(v^{k,\beta,\varepsilon}_{bl}\bigl(\frac{x}{\varepsilon}\bigr)-V^{k,\beta,*}\Bigr)\partial_{x_\alpha}\partial_{x_\beta}u^0
\end{Vmatrix}_{L^2(\Omega)}\leq 
\begin{Vmatrix}
v^{k,\beta,\varepsilon}_{bl}\bigl(\frac{x}{\varepsilon}\bigr)-V^{k,\beta,*}
\end{Vmatrix}_{L^2(\Omega)}
\begin{Vmatrix}
\partial_{x_\alpha}\partial_{x_\beta}u^0
\end{Vmatrix}_{L^\infty(\Omega)}.
\end{equation*}
The assumption $u^0\in C^2(\overline{\Omega})$ plays here the same role as $u^0\in C^1(\overline{\Omega})$ for \eqref{eqmajvblkepsm}. Use \eqref{ineqvblkepsalphaLp} with $p=2$ to conclude.

If $u^0\in H^{2+\omega}(\Omega)$, we cannot proceed as above because $\partial_{x_\alpha}\partial_{x_\beta}u^0$ does not belong to $L^\infty(\Omega)$. By the Sobolev injections, $H^\omega(\Omega)$ is continuously embedded in $L^q(\Omega)$ for all $1\leq q<\frac{2}{1-\omega}$. Yet $\partial_{x_\alpha}\partial_{x_\beta}u^0$ is in $H^\omega(\Omega)$. Take $2\leq q<\frac{2}{1-\omega}$ and $\widehat{q}\geq 2$ such that $\frac{1}{q}+\frac{1}{\widehat{q}}=\frac{1}{2}$. Necessarily $\frac{2}{\omega}<\widehat{q}$. H\"older's inequality yields
\begin{equation*}
\begin{Vmatrix}
\Bigl(v^{k,\beta,\varepsilon}_{bl}\bigl(\frac{x}{\varepsilon}\bigr)-V^{k,\beta,*}\Bigr)\partial_{x_\alpha}\partial_{x_\beta}u^0
\end{Vmatrix}_{L^2(\Omega)}\\
\leq 
\begin{Vmatrix}
v^{k,\beta,\varepsilon}_{bl}\bigl(\frac{x}{\varepsilon}\bigr)-V^{k,\beta,*}
\end{Vmatrix}_{L^{\widehat{q}}(\Omega)}
\begin{Vmatrix}
\partial_{x_\alpha}\partial_{x_\beta}u^0
\end{Vmatrix}_{L^q(\Omega)}.
\end{equation*}
Apply now \eqref{ineqvblkepsalphaLp} with $p=\widehat{q}$ to get 
$\begin{Vmatrix}
\nabla_xv^{k,\varepsilon}_{bl}\bigl(x,\frac{x}{\varepsilon}\bigr)
\end{Vmatrix}_{L^2(\Omega)}=O\bigl(\varepsilon^\frac{1}{\widehat{q}}\bigr)$.

The second term of $r_{bl}^{\varepsilon,k}$ needs to be treated differently. The key ingredient is Hardy's inequality: for all $\phi\in H^1_0(\Omega)$
\begin{equation*}
\begin{Vmatrix}
\frac{\phi(x)}{d(x,\partial\Omega)}
\end{Vmatrix}_{L^2(\Omega)}
\leq \begin{Vmatrix}
\nabla\phi
\end{Vmatrix}_{L^2(\Omega)}
\end{equation*}
where $d(x,\partial\Omega)$ is the distance from $x$ to $\partial\Omega$. Let $\phi\in H^1_0(\Omega)$. 
For all $2\leq q,\widehat{q}\leq\infty$ such that $\frac{1}{q}+\frac{1}{\widehat{q}}=2$,
\begin{align*}
&\left|\int_{\Omega}\Bigl[\nabla_x\cdot A(y)\nabla_yv_{bl}^{k,\varepsilon}\Bigr]\Bigl(x,\frac{x}{\varepsilon}\Bigr)\phi(x)dx\right|\\
&\qquad\qquad\leq C\varepsilon\begin{Vmatrix}
\partial_{y_\beta}v_{bl}^{k,\gamma,\varepsilon}\bigl(\frac{x}{\varepsilon}\bigr)\partial_{x_\alpha}\partial_{x_\gamma}u^0(x)\frac{\left|x\cdot n^k-c^k\right|}{\varepsilon}
\end{Vmatrix}_{L^2(\Omega)}\begin{Vmatrix}
\frac{\phi(x)}{d(x,\partial\Omega)}
\end{Vmatrix}_{L^2(\Omega)}\\
&\qquad\qquad\leq C\varepsilon\begin{Vmatrix}
\partial_{y_\beta}v_{bl}^{k,\gamma,\varepsilon}\bigl(\frac{x}{\varepsilon}\bigr)\frac{\left|x\cdot n^k-c^k\right|}{\varepsilon}
\end{Vmatrix}_{L^{\widehat{q}}(\Omega)}
\begin{Vmatrix}
\partial_{x_\alpha}\partial_{x_\gamma}u^0
\end{Vmatrix}_{L^{q}(\Omega)}\begin{Vmatrix}
\nabla\phi
\end{Vmatrix}_{L^2(\Omega)}.
\end{align*}
If $u^0\in H^{2+\omega}(\Omega)$, then take $2\leq q<\frac{2}{1-\omega}$ and apply \eqref{ineqvblkepsalphaLp} with $p=\widehat{q}$ and $m=1$. If $u^0\in C^2(\overline{\Omega})$, then take $q=\infty$ and $\widehat{q}=2$ and apply \eqref{ineqvblkepsalphaLp} with $p=2$ and $m=1$.
\end{proof}

\begin{lem}\label{lemphiblepsH1/2}
The boundary function $\varphi_{bl}^\varepsilon$ is
\begin{enumerate}
\item of order $O\bigl(\varepsilon^{\omega})$ in $W^{1-\frac{1}{p},p}(\partial\Omega)$ for all $1\leq p<2$, if $u^0\in H^{2+\omega}(\Omega)$;
\item of order $O(\varepsilon)$ in $W^{1-\frac{1}{p},p}(\partial\Omega)$ for all $1\leq p<2$, if $u^0\in H^3(\Omega)\cap C^2(\overline{\Omega})$.
\end{enumerate}
\end{lem}

\begin{proof}[Proof]
First of all, using proposition \ref{propboundarydatahomoH12} one notices that $\varphi_{bl}^\varepsilon$ belongs to $H^\frac{1}{2}(\partial\Omega)$. 
As $\varphi_{bl}^\varepsilon$ factors into $V\bigl(\frac{\cdot}{\varepsilon}\bigr)\nabla u^0$ we immediatly get the very rough estimate
\begin{equation*}
\begin{Vmatrix}
\varphi^\varepsilon_{bl}
\end{Vmatrix}_{H^\frac{1}{2}(\partial\Omega)}=O\bigl(\varepsilon^{-\frac{1}{2}}\bigr)
\end{equation*}
which is far from being enough. We do not try further to get a bound in $H^\frac{1}{2}(\partial\Omega)$.

We refer to \cite{dgvnm} for the case when $u^0\in H^3(\Omega)\cap C^2(\overline{\Omega})$. If $u^0\in H^{2+\omega}(\Omega)$ the proof follows the same scheme, with differences due to the weaker regularity assumption on $u^0$. \emph{Assume for the rest of the proof that $u^0\in H^{2+\omega}(\Omega)$.} The edge estimate goes on as in the case $u^0\in H^3(\Omega)\cap C^2(\overline{\Omega})$ and one gets for all $1\leq p<2$, $m\in\mathbb N$
\begin{equation*}
\begin{Vmatrix}
\psi\varphi^\varepsilon
\end{Vmatrix}_{W^{1-\frac{1}{p},p}(\partial\Omega)}=O\bigl(\varepsilon^m\bigr)
\end{equation*}
where $\psi$ is a smooth function on $\partial\Omega$ compactly supported in $\partial\Omega\cap K^k$ outside the vertices.

Let us now focus on the estimate near a vertex $O$ lying at the intersection of $K^1$ and $K^2$. We introduce polar coordinates $r=r(x)$ and $\theta=\theta(x)$ centered at $O$ and use a smooth function $\psi$ on $\partial\Omega$ compactly supported in a vicinity of $O$. Let $1\leq p$. The tame estimate
\begin{equation}\label{tameest}
\begin{Vmatrix}
fg
\end{Vmatrix}_{W^{1-\frac{1}{p},p}(\partial\Omega)}
\leq
C\left(
\begin{Vmatrix}
f
\end{Vmatrix}_{L^\infty(\partial\Omega)}
\begin{Vmatrix}
g
\end{Vmatrix}_{W^{1-\frac{1}{p},p}(\partial\Omega)}
+\begin{Vmatrix}
g
\end{Vmatrix}_{L^\infty(\partial\Omega)}
\begin{Vmatrix}
f
\end{Vmatrix}_{W^{1-\frac{1}{p},p}(\partial\Omega)}
\right)
\end{equation}
holds for all $f,g\in L^\infty(\partial\Omega)\cap W^{1-\frac{1}{p},p}(\partial\Omega)$.
Taking advantage of the fact that $H^{2+\omega}(\Omega)$ injects in $C^{1,\omega}(\overline{\Omega})$, one knows $\frac{\nabla u^0}{r^\omega}\in L^\infty(\partial\Omega)$. Besides, $\frac{\nabla u^0}{r^\omega}$ belongs to $W^{1,p}(\Omega)$ for all $1\leq p<2$. Therefore $\frac{\nabla u^0}{r^\omega}\in L^\infty(\partial\Omega)\cap W^{1-\frac{1}{p},p}(\partial\Omega)$ for all $1\leq p<2$ and \eqref{tameest} yields 
\begin{multline*}
\begin{Vmatrix}
\psi^2\varphi^\varepsilon_{bl}
\end{Vmatrix}_{W^{1-\frac{1}{p},p}(\partial\Omega)}
\leq
C\biggl(
\begin{Vmatrix}
\psi r^\omega V\bigl(\frac{\cdot}{\varepsilon}\bigr)
\end{Vmatrix}_{W^{1-\frac{1}{p},p}(\partial\Omega)}
\begin{Vmatrix}
\psi\frac{\nabla u^0}{r^\omega}
\end{Vmatrix}_{L^\infty(\partial\Omega)}\\
+\begin{Vmatrix}
\psi r^\omega V\bigl(\frac{\cdot}{\varepsilon}\bigr)
\end{Vmatrix}_{L^\infty(\partial\Omega)}
\begin{Vmatrix}
\psi\frac{\nabla u^0}{r^\omega}
\end{Vmatrix}_{W^{1-\frac{1}{p},p}(\partial\Omega)}
\biggr).
\end{multline*}
By estimating first on $\partial\Omega\cap K^1$ then on $\partial\Omega\cap K^2$ one obtains for all $1\leq p<2$
\begin{subequations}
\begin{align}
&\begin{Vmatrix}
\psi r^\omega V\bigl(\frac{\cdot}{\varepsilon}\bigr)
\end{Vmatrix}_{L^\infty(\partial\Omega)}=O\bigl(\varepsilon^\omega\bigr)\nonumber\\
&\begin{Vmatrix}
\psi r^\omega V\bigl(\frac{\cdot}{\varepsilon}\bigr)
\end{Vmatrix}_{L^p(\partial\Omega)}=O\bigl(\varepsilon^{\omega+\frac{1}{p}}\bigr)\label{estinterLpromegaV}\\
&\begin{Vmatrix}
\psi r^\omega V\bigl(\frac{\cdot}{\varepsilon}\bigr)
\end{Vmatrix}_{W^{1,p}(\partial\Omega)}=O\bigl(\varepsilon^{\omega-1+\frac{1}{p}}\bigr)\label{estinterW1promegaV}.
\end{align}
\end{subequations}
Interpolating \eqref{estinterLpromegaV} and \eqref{estinterW1promegaV} gives
\begin{equation*}
\begin{Vmatrix}
\psi r^\omega V\bigl(\frac{\cdot}{\varepsilon}\bigr)
\end{Vmatrix}_{W^{1-\frac{1}{p},p}(\partial\Omega)}
\leq C
\begin{Vmatrix}
\psi r^\omega V\bigl(\frac{\cdot}{\varepsilon}\bigr)
\end{Vmatrix}_{W^{1,p}(\partial\Omega)}^{1-\frac{1}{p}}
\begin{Vmatrix}
\psi r^\omega V\bigl(\frac{\cdot}{\varepsilon}\bigr)
\end{Vmatrix}_{L^p(\partial\Omega)}^\frac{1}{p}=O\bigl(\varepsilon^{\omega-1+\frac{2}{p}}\bigr).
\end{equation*}
Finally, $\begin{Vmatrix}\psi^2\varphi^\varepsilon_{bl}\end{Vmatrix}_{W^{1-\frac{1}{p},p}}=O\bigl(\varepsilon^{\omega}\bigr)$ which concludes our proof.
\end{proof}

We conclude this section by expounding how to deduce a bound on $e^\varepsilon_{bl}$ from the lemmas \ref{lemrblepsH-1} and \ref{lemphiblepsH1/2}. We focus on the case when $u^0\in H^{2+\omega}(\Omega)$, as the reasoning is a little more subtle than in the case $u^0\in H^3(\Omega)\cap C^2(\overline{\Omega})$. It is very straightforward to adapt the proof in the latter case (see also \cite{dgvnm}). Lemma \ref{lemphiblepsH1/2} gives bounds for $\varphi^\varepsilon_{bl}$ in $W^{1-\frac{1}{p},p}(\partial\Omega)$ for $1\leq p<2$. As we lack an estimate in $H^\frac{1}{2}(\partial\Omega)$, we cannot bound $e^\varepsilon_{bl}$ in $H^1(\Omega)$. We thus have to use results on elliptic equations in divergence form and with source term in some $W^{-1,p}(\Omega)$ space. Let us state a general theorem that suits to our framework (for references see below).
\begin{theo}[Meyers]\label{theomeyersW1p}
Let $\Omega\subset\mathbb R^d$ be a Lipschitz domain, $A=A^{\alpha\beta}(y)\in M_N(\mathbb R)$ a family of $C^\infty(\overline{\Omega})$ functions. Assume the ellipticity of $A$.\\
There exists a $p_0<2$ such that for all $f\in H^{-1}(\Omega)$ if $u\in H^1_0(\Omega)$ is a weak solution of $-\nabla\cdot A\nabla u=f$ in $H^{-1}(\Omega)$ and if for all $p_0< p<2$, $f\in W^{-1,p}(\Omega)$, then $u\in W^{1,p}_0(\Omega)$ and there exists $C(p)>0$,
\begin{equation*}
\begin{Vmatrix}
u
\end{Vmatrix}_{W^{1,p}_0(\Omega)}\leq C(p)
\begin{Vmatrix}
f
\end{Vmatrix}_{W^{-1,p}(\Omega)}.
\end{equation*}
\end{theo}
Such an estimate originally appeared in the work of Meyers \cite{norm}, where the case of smooth $C^2$ domains $\Omega$ and scalar equations is treated. It has been extended by Gallouet and Monier in \cite{gallouet} to domains $\Omega$ with Lipschitz boundary. 
In their recent survey article \cite{mazya}, Maz'ya and Shaposhnikova give very general estimates working for Lipschitz domains $\Omega$ and systems of elliptic equations. Our theorem \ref{theomeyersW1p} happens to be a very special case of theorems $1$ and $2$ in \cite{mazya}. To make the link obvious take $m=1$, $l=N$, $a=0$; then $s=1-\frac{1}{p}$, $W^{m,a}_p(\Omega)=W^{1,p}(\Omega)$ and $V^a_p(\Omega)=W^{1,p}_0(\Omega)$. 

It is important to notice that $p_0$ (resp. $C(p)$) only depends on the coercivity constant of $A$ (resp. on the coercivity constant of $A$ and $p$). This makes the theorem applicable to our homogenization problem. We know from the proof of lemma \ref{lemphiblepsH1/2} that 
\begin{equation*}
\varphi_{bl}^\varepsilon\in \bigcap_{1\leq p\leq 2}W^{1-\frac{1}{p},p}(\partial\Omega).
\end{equation*}
Thus there exists a lifting $\phi^\varepsilon_{bl}$ of $\varphi_{bl}^\varepsilon$ belonging to $W^{1,p}(\Omega)$ for all $1\leq p\leq 2$ such that 
\begin{equation*}
\begin{Vmatrix}
\phi^\varepsilon_{bl}
\end{Vmatrix}_{W^{1,p}(\Omega)}\leq C(p)
\begin{Vmatrix}
\varphi^\varepsilon_{bl}
\end{Vmatrix}_{W^{1-\frac{1}{p},p}(\partial\Omega)}
\end{equation*}
with $C(p)$ independent of $\varepsilon$, as usual.
The difference $e^\varepsilon_{bl}-\phi_{bl}^\varepsilon\in H^1_0(\Omega)$ solves 
\begin{equation*}
-\nabla\cdot A\Bigl(\frac{x}{\varepsilon}\Bigr)\nabla\left(e^\varepsilon_{bl}-\phi_{bl}^\varepsilon\right)=r^\varepsilon_{bl}+\nabla\cdot A\Bigl(\frac{x}{\varepsilon}\Bigr)\nabla\phi_{bl}^\varepsilon=:F^\varepsilon_{bl}.
\end{equation*} 
As $F^\varepsilon_{bl}$ belongs to $W^{-1,p}(\Omega)$ for $1\leq p\leq 2$, we have 
\begin{equation}\label{estW-1pFepsbl}
\begin{Vmatrix}
F^\varepsilon_{bl}
\end{Vmatrix}_{W^{-1,p}(\Omega)}\leq C(p)\left[
\begin{Vmatrix}
r^\varepsilon_{bl}
\end{Vmatrix}_{H^{-1}(\Omega)}+
\begin{Vmatrix}
\varphi^\varepsilon_{bl}
\end{Vmatrix}_{W^{1-\frac{1}{p},p}(\partial\Omega)}\right].
\end{equation}
Let $p_0<2$ given by theorem \ref{theomeyersW1p}. Then, for all $p_0<p\leq 2$, 
\begin{equation*}
\begin{Vmatrix}
e^\varepsilon_{bl}-\phi_{bl}^\varepsilon
\end{Vmatrix}_{W^{1,p}_0(\Omega)}\leq C(p)
\begin{Vmatrix}
F^\varepsilon_{bl}
\end{Vmatrix}_{W^{-1,p}(\Omega)}.
\end{equation*}
Let $0<\gamma<\frac{\omega}{2}$. Then, it follows from \eqref{estW-1pFepsbl} and from lemmas \ref{lemrblepsH-1} and \ref{lemphiblepsH1/2} that for all $p_0<p<2$,
\begin{equation*}
\begin{Vmatrix}
F^\varepsilon_{bl}
\end{Vmatrix}_{W^{-1,p}(\Omega)}=O\bigl(\varepsilon^\gamma\bigr).
\end{equation*}
To get an $L^2(\Omega)$ estimate on $e^\varepsilon_{bl}$ use the Sobolev injection of $W^{1,p}(\Omega)$ in $L^2(\Omega)$ and, once again, our $W^{1-\frac{1}{p},p}(\partial\Omega)$ bound on $\varphi^\varepsilon_{bl}$. 

\selectlanguage{english}

\section{A first-order asymptotic expansion of the eigenvalues}
\label{secasy}

This section is concerned with the final step of the proof of theorems \ref{theoasylisse} and \ref{theoasypol}.  
Let $E_{\lambda^0}$ be the finite-dimensional eigenspace associated to $\lambda^0$. From the ideas explained in the introduction, and in particular the third part of theorem \ref{theoregdaugekmr}, we know, in any case, that $E_{\lambda^0}\subset H^{2+\omega}(\Omega)$, with $0<\omega$. When $\Omega$ is a smooth uniformly convex domain, we take $\omega=1$.

We have recourse to the ideas involved in \cite{moscovog} to prove the asymptotic expansion of the eigenvalues. Moskow and Vogelius use abstract estimates due to Osborn in \cite{osborn}. We recall the estimate we need in terms of $T^\varepsilon$ and $T^0$. Assume that $\lambda^0$ is an eigenvalue of order $m$. Then, $E_{\lambda^0}$ is $m$-dimensional. Let $\lambda^0=\lambda^{0,k}=\lambda^{0,k+1}=\ldots=\lambda^{0,k+m-1}$. The associated eigenvectors $v^{0,k},\ldots,v^{0,k+m-1}$ form an orthogonal basis of $E_{\lambda^0}$.
\begin{theo}[Osborn in \cite{osborn}]
There exists a constant $C>0$ such that
\begin{multline}\label{estosborn}
\begin{vmatrix}
\frac{1}{\lambda^0}-\frac{1}{m}\sum_{j=0}^{m-1}\frac{1}{\lambda^{\varepsilon,k+j}}-\frac{1}{m}\sum_{j=0}^{m-1}\left\langle(T^{\varepsilon}-T^0)v^{0,k+j},v^{0,k+j}\right\rangle
\end{vmatrix}
\leq C\begin{Vmatrix}
(T^{\varepsilon}-T^0)|_{E_{\lambda^0}}
\end{Vmatrix}^2,
\end{multline}
where $T^{\varepsilon}$ and $T^0$ are seen as operators acting in $L^2(\Omega)$ and $\langle\cdot,\cdot\rangle$ denotes the scalar product in $L^2(\Omega)$.
\end{theo}
This theorem is a straightforward corollary of theorem $3.1$ in \cite{moscovog}. Its proof really uses all properties of the operators $T^\varepsilon$ and $T^0$, among other things selfadjointness and compactness.


The first thing to do is to estimate $\begin{Vmatrix}
(T^{\varepsilon}-T^0)|_{E_{\lambda^0}}
\end{Vmatrix}$. Let $f\in E_{\lambda^0}$; we call $u^\varepsilon:=T^{\varepsilon}f$ and $u^0:=T^0f$. We need to estimate $\begin{Vmatrix}u^\varepsilon-u^0\end{Vmatrix}_{L^2(\Omega)}$. We can improve the bounds of section \ref{secsomerrorest}, such as \eqref{inequepsu0L2}. Those bounds are not enough to deduce from \eqref{estosborn} a first-order asymptotic expansion of $\lambda^\varepsilon$. The loss of $O\left(\frac{1}{\sqrt{\varepsilon}}\right)$ in estimate \eqref{inequepsu0L2} is due to the bad bound of $\vartheta_{u,bl}^\varepsilon$ in $H^{\frac{1}{2}}(\partial\Omega)$:
\begin{equation*}
\begin{Vmatrix}
\vartheta_{u,bl}^{\varepsilon}
\end{Vmatrix}_{L^2(\Omega)}
\leq
\begin{Vmatrix}
\vartheta_{u,bl}^{\varepsilon}
\end{Vmatrix}_{H^1(\Omega)}
\leq C
\begin{Vmatrix}
\chi^\alpha\bigl(\frac{x}{\varepsilon}\bigr)\partial_{x_\alpha}u^0(x)
\end{Vmatrix}_{H^\frac{1}{2}(\partial\Omega)}
\leq C\varepsilon^{-\frac{1}{2}}
\begin{Vmatrix}
u^0
\end{Vmatrix}_{H^2(\Omega)}.
\end{equation*}

If $\Omega$ is a smooth domain, \eqref{estuepsu0epsL2} can be shown thanks to the results of Avellaneda and Lin. Theorem \ref{theoavlinsys} yields indeed
\begin{equation*}
\begin{Vmatrix}
\vartheta_{u,bl}^{\varepsilon}
\end{Vmatrix}_{L^2(\Omega)}\leq C\begin{Vmatrix}
\varphi\bigl(\cdot,\frac{\cdot}{\varepsilon}\bigr)\end{Vmatrix}_{L^2(\partial\Omega)}\leq C\begin{Vmatrix}
u^0
\end{Vmatrix}_{H^2(\Omega)}.
\end{equation*}
The assumption $u^0\in H^2(\Omega)$ being clearly fulfilled as $u^0=T^0f=\frac{1}{\lambda^0}f$, it is easy to adapt the proof of corollary \ref{coruepsu0L2} to conclude that:
\begin{equation}\label{estuepsu0epsL2sys}
\begin{Vmatrix}
u^\varepsilon-u^0
\end{Vmatrix}_{L^2(\Omega)}
\leq C\varepsilon\begin{Vmatrix}
u^0
\end{Vmatrix}_{H^2(\Omega)}
\leq C\varepsilon\begin{Vmatrix}
u^0
\end{Vmatrix}_{H^{2+\omega}(\Omega)}.
\end{equation}

Assume now that $\Omega$ is a polygonal domain satisfying either {\bf (RAT)} or {\bf (DIV)}. A uniform bound in $\varepsilon$ of $\vartheta_{u,bl}^{\varepsilon}$ does not follow from the results of Avellaneda and Lin. The estimates of section \ref{secsomerrorest} are sufficient to get the convergence of the boundary layer in section \ref{sechomblsys}, up to the extraction of a subsequence under assumption {\bf (RAT)}. Actually theorem \ref{CVhomorat} (resp. \ref{CVhomodiv}) implies that: there exists a sequence $(\varepsilon_n)$ such that $\begin{Vmatrix}
\vartheta_{u,bl}^{\varepsilon_n}
\end{Vmatrix}_{L^2(\Omega)}\leq C\begin{Vmatrix}
u^0
\end{Vmatrix}_{H^{2+\omega}(\Omega)}$
(resp. for all $0<\varepsilon$, $\begin{Vmatrix}
\vartheta_{u,bl}^{\varepsilon}
\end{Vmatrix}_{L^2(\Omega)}\leq C\begin{Vmatrix}
u^0
\end{Vmatrix}_{H^{2+\omega}(\Omega)}$).
We conclude, as in the case when $\Omega$ is smooth, that \eqref{estuepsu0epsL2sys} holds. In order to avoid extracting subsequences we now omit the case of polygonal domains under assumption {\bf (RAT)}. 

It remains to bound $\begin{Vmatrix}
u^0
\end{Vmatrix}_{H^{2+\omega}(\Omega)}$ by $\begin{Vmatrix}f\end{Vmatrix}_{L^2(\Omega)}$. Taking advantage of the equivalence of norms on the finite-dimensional space $E_{\lambda^0}\subset H^{2+\omega}(\Omega)\subset H^2(\Omega)$, there exists $0<C$ such that for all $w\in E_{\lambda^0}$,
\begin{equation}\label{esteqnorm}
\begin{Vmatrix}
w
\end{Vmatrix}_{H^{2+\omega}(\Omega)}
\leq C
\begin{Vmatrix}
w
\end{Vmatrix}_{L^2(\Omega)}.
\end{equation}
Therefore, combining \eqref{estuepsu0epsL2sys} with \eqref{esteqnorm}, we get
\begin{equation*}
\begin{Vmatrix}
T^{\varepsilon}f-T^0f
\end{Vmatrix}_{L^2(\Omega)}\leq C\varepsilon
\begin{Vmatrix}
u^0
\end{Vmatrix}_{H^{2+\omega}(\Omega)}
\leq C\varepsilon
\begin{Vmatrix}f\end{Vmatrix}_{L^2(\Omega)}
\end{equation*}
which shows that 
\begin{equation*}
\begin{Vmatrix}
(T^{\varepsilon}-T^0)|_{E_{\lambda^0}}
\end{Vmatrix}\leq C\varepsilon.
\end{equation*}

Our final goal is to prove \eqref{devptasylisse}, \eqref{devptasypol} and \eqref{devptasypolreg}. The reasoning, in every case, follows the lines of \cite{moscovog}. Estimate \eqref{estosborn} now sums up in:
\begin{equation}\label{dvptasyinter}
\frac{1}{\lambda^0}-\frac{1}{m}\sum_{j=0}^{m-1}\frac{1}{\lambda^{\varepsilon,k+j}}=\frac{1}{m}\sum_{j=0}^{m-1}\left\langle(T^{\varepsilon}-T^0)v^{0,k+j},v^{0,k+j}\right\rangle
+O\bigl(\varepsilon^2\bigr).
\end{equation}
Let us focus on $\frac{1}{m}\sum_{j=0}^{m-1}\left\langle(T^{\varepsilon}-T^0)v^{0,k+j},v^{0,k+j}\right\rangle$ and work on $\left\langle(T^{\varepsilon}-T^0)v^{0,k+j},v^{0,k+j}\right\rangle$, $j$ being fixed in $\{0,\ldots,m-1\}$. We call $u^{\varepsilon,k+j}:=T^{\varepsilon}v^{0,k+j}$. This function solves \eqref{sysosc}. According to estimate \eqref{enestL2ineqomega} of theorem \ref{theoL2ineqomega}, as $v^{0,k+j}\in H^{2+\omega}(\Omega)$, 
\begin{equation*}
\begin{Vmatrix}
u^{\varepsilon,k+j}(x)-\frac{1}{\lambda^0}v^{0,k+j}(x)-\frac{\varepsilon}{\lambda^0}\chi^\alpha\left(\frac{x}{\varepsilon}\right)\partial_{x_\alpha}v^{0,k+j}(x)+\frac{\varepsilon}{\lambda^0}\vartheta^{\varepsilon}_{v,k+j,bl}(x)
\end{Vmatrix}_{L^2(\Omega)}=O\bigl(\varepsilon^{1+\frac{\omega}{2}}\bigr),
\end{equation*}
where $\vartheta^{\varepsilon}_{v,k+j,bl}$ solves \eqref{sysvartheta} with $v^{0,k+j}$ instead of $u^0$. Cauchy-Schwarz inequality implies
\begin{multline}\label{dvptasyordre1}
\left\langle(T^{\varepsilon}-T^0)v^{0,k+j},v^{0,k+j}\right\rangle=\int_{\Omega}\left(\frac{1}{\lambda^0}v^{0,k+j}(x)-u^{\varepsilon,k+j}(x)\right)v^{0,k+j}(x)dx\\
=\frac{\varepsilon}{\lambda^0}\int_{\Omega}\chi^\alpha\left(\frac{x}{\varepsilon}\right)\partial_{x_\alpha}v^{0,k+j}(x)\cdot v^{0,k+j}(x)dx+\frac{\varepsilon}{\lambda^0}\int_{\Omega}\vartheta^{\varepsilon}_{v,k+j,bl}(x)\cdot v^{0,k+j}(x)dx+O\bigl(\varepsilon^{1+\frac{\omega}{2}}\bigr).
\end{multline}
We intend to show that the term involving $\chi^\alpha$ in \eqref{dvptasyordre1} is of order $O\bigl(\varepsilon^{1+\frac{\omega}{2}}\bigr)$. In order to carry out integrations by parts, we introduce, for each $1\leq\alpha\leq 2$, a periodic $C^\infty$ solution $b^\alpha=b^\alpha(y)\in M_N(\mathbb R)$ to
\begin{equation*}
\Delta_yb^{\alpha}=\chi^\alpha,
\end{equation*}
the Fredholm property being satisfied as $\int_{\mathbb T^2}\chi^\alpha(y)dy=0$. An integration by part gives
\begin{align*}
&\frac{\varepsilon}{\lambda^0}\int_{\Omega}\chi^\alpha\left(\frac{x}{\varepsilon}\right)\partial_{x_\alpha}v^{0,k+j}(x)\cdot v^{0,k+j}(x)dx\\
&\qquad\qquad=\frac{\varepsilon}{\lambda^0}\int_{\Omega}\varepsilon^2\Delta\left(b^\alpha\left(\frac{x}{\varepsilon}\right)\right)\partial_{x_\alpha}v^{0,k+j}(x)\cdot v^{0,k+j}(x)dx\\
&\qquad\qquad=-\frac{\varepsilon^2}{\lambda^0}\int_{\Omega}\varepsilon\nabla\left(b^\alpha\left(\frac{x}{\varepsilon}\right)\right)\cdot\nabla\left(\partial_{x_\alpha}v^{0,k+j}(x)v^{0,k+j}(x)\right)dx\\
&\qquad\qquad\leq C\varepsilon^2\begin{Vmatrix}
\varepsilon\nabla\left(b^\alpha\left(\frac{x}{\varepsilon}\right)\right)
\end{Vmatrix}_{L^\infty(\Omega)}
\begin{Vmatrix}
\nabla\left(\partial_{x_\alpha}v^{0,k+j}(x)v^{0,k+j}(x)\right)
\end{Vmatrix}_{L^1(\Omega)}\\
&\qquad\qquad\leq C\varepsilon^2.
\end{align*}
We deduce from \eqref{dvptasyinter} and \eqref{dvptasyordre1} that
\begin{align*}
\frac{1}{\lambda^0}-\frac{1}{m}\sum_{j=0}^{m-1}\frac{1}{\lambda^{\varepsilon,k+j}}
&=\frac{1}{m}\sum_{j=0}^{m-1}\left\langle(T^{\varepsilon}-T^0)v^{0,k+j},v^{0,k+j}\right\rangle
+O\bigl(\varepsilon^2\bigr)\\
&=\frac{1}{m}\sum_{j=0}^{m-1}\frac{\varepsilon}{\lambda^0}\int_{\Omega}\vartheta^{\varepsilon}_{v,k+j,bl}(x)\cdot v^{0,k+j}(x)dx+O\bigl(\varepsilon^{1+\frac{\omega}{2}}\bigr).
\end{align*}

The results of section \ref{sechomblsys} now apply, in particular theorems \ref{theodgvnm2}, \ref{CVhomorat} (in this case up to the extraction of a subsequence) and \ref{CVhomodiv}, and yield that
\begin{equation*}
\begin{Vmatrix}
\vartheta_{v,k+j,bl}^{\varepsilon}-\vartheta_{v,k+j,bl}^*
\end{Vmatrix}_{L^2(\Omega)}=O\bigl(\varepsilon^{\gamma}\bigr)
\end{equation*}
for suitable exponents $0<\gamma$: 
\begin{enumerate}
\item for all $0\leq\gamma<\frac{1}{11}$, when $\Omega$ is a smooth uniformly convex domain;
\item for all $0\leq\gamma<\frac{\omega}{2}$ (resp. for $\gamma=\frac{1}{2}$), when $\Omega$ is a convex polygon satisfying either {\bf (RAT)} or {\bf (DIV)} and  $E_{\lambda^0}\subset H^{2+\omega}(\Omega)$ (resp. $E_{\lambda^0}\subset H^{3}(\Omega)\cap C^2(\overline{\Omega})$).
\end{enumerate}
Therefore,
\begin{equation*}
\frac{1}{\lambda^0}-\frac{1}{m}\sum_{j=0}^{m-1}\frac{1}{\lambda^{\varepsilon,k+j}}
=\varepsilon\frac{1}{m\lambda^0}\sum_{j=0}^{m-1}\int_{\Omega}\vartheta_{v,k+j,bl}^*(x)\cdot v^{0,k+j}(x)dx+O\bigl(\varepsilon^{1+\gamma}\bigr),
\end{equation*}
with $\gamma$ given above; 
from the convergence of the eigenvalues $\lambda^{\varepsilon,k+j}$ towards $\lambda^{0,k+j}$, we deduce that
\begin{equation*}
\left[\frac{1}{m}\sum_{j=0}^{m-1}\frac{1}{\lambda^{\varepsilon,k+j}}\right]^{-1}=\lambda^0
+\varepsilon\frac{\lambda^0}{m}\sum_{j=0}^{m-1}\int_{\Omega}\vartheta_{v,k+j,bl}^*(x)\cdot v^{0,k+j}(x)dx+O\bigl(\varepsilon^{1+\gamma}\bigr),
\end{equation*}
which achieves the proof of theorems \ref{theoasylisse} and \ref{theoasypol}. 

\section*{Acknowledgement}
The research of this paper was supported by the Agence Nationale de la Recherche under the grant ANR-$08$-JCC-$0104$ coordinated by David G\'erard-Varet. The author would like to thank his PHD advisor David G\'erard-Varet for bringing this subject to him.

\nocite{*}

\bibliographystyle{plain} 
\bibliography{homovp5_bib} 




\end{document}